\documentclass[DIV=12]{scrartcl}

\usepackage{lmodern}
\usepackage[utf8]{inputenc}
\usepackage[T1]{fontenc}
\usepackage[english]{babel}
\usepackage{csquotes}
\usepackage{subcaption}
\usepackage[backend=biber,style=alphabetic,url=false,giveninits=true,sorting=anyt]{biblatex}
\usepackage{float}
\usepackage{amsmath,amssymb,amsthm}
\usepackage{mathtools,thmtools,stmaryrd,esint}
\usepackage{mathrsfs}

\usepackage{hyperref}
\hypersetup{colorlinks=true}

\usepackage[nameinlink]{cleveref}
\crefname{equation}{}{f}

\usepackage[svgnames]{xcolor}
\usepackage{graphicx}
\usepackage[shortlabels]{enumitem}

\usepackage[disable]{todonotes}
\usepackage{bm} 
\usepackage{bbold}
\usepackage{eucal}

\numberwithin{equation}{section}
\makeatletter
\newcommand{\mylabel}[2]{#2\def\@currentlabel{#2}\label{#1}}
\makeatother

\declaretheorem[style=plain,numberwithin=section]{theorem}
\declaretheorem[style=plain,numberlike=theorem]{lemma, proposition, corollary}
\declaretheorem[style=definition,numberwithin=section]{definition}
\declaretheorem[style=remark,numberlike=definition]{remark,example}

\declaretheorem[name=Theorem A,numbered=no,style=plain]{theoremA}
\declaretheorem[name=Theorem B,numbered=no,style=plain]{theoremB}

\newtheorem*{theorem*}{Theorem}

\newcommand{\myref}[1]{\textnormal{\ref{#1}}}

\newcommand{\N}{\mathbb{N}}

\newcommand{\R}{\mathbb{R}}

\newcommand{\measspace}{\mathscr{M}}

\newcommand{\eps}{\varepsilon}
\newcommand{\hdm}{{\mathscr H}}

\newcommand{\one}{\mathbf{1}}
\newcommand{\dd}{\mathop{}\mathopen{}\mathrm{d}}

\newcommand{\defeq}{\stackrel{\text{def}}{=}}

\DeclareMathOperator{\spn}{span}

\DeclareMathOperator{\entropy}{Ent}
\DeclareMathOperator{\diagonal}{diag}
\DeclareMathOperator{\domain}{dom}
\DeclareMathOperator{\interior}{int}
\DeclareMathOperator{\F}{\mathrm{F}}

\DeclarePairedDelimiter\abs{\lvert}{\rvert}
\DeclarePairedDelimiter\norm{\lVert}{\rVert}
\DeclarePairedDelimiterX\set[2]{\{}{\}}{#1\,\delimsize\vert\,#2}
\DeclarePairedDelimiterX\inp[2]{\langle}{\rangle}%
{#1\,\delimsize\vert\,\mathopen{}#2}

\title{Convergence rates for regularized unbalanced optimal transport: the discrete case}
\author{Luca Nenna\thanks{Universit\'e Paris-Saclay, CNRS, Laboratoire de math\'ematiques d'Orsay, ParMA, Inria Saclay, 91405, Orsay, France. 
email: \texttt{luca.nenna@universite-paris-saclay.fr}} \thanks{Institut Universitaire de France (IUF).} ,\; Paul Pegon\thanks{CEREMADE, Université Paris--Dauphine, Université PSL, CNRS \& MOKAPLAN, Inria Paris, 75016 Paris, France. Email: \texttt{pegon@ceremade.dauphine.fr}} \; and Louis Tocquec\thanks{Universit\'e Paris-Saclay, CNRS, Laboratoire de math\'ematiques d'Orsay \& MOKAPLAN, Inria Paris, 75016 Paris, France \& ParMA, Inria Saclay, 91405, Orsay, France. Email: \texttt{louis.tocquec@universite-paris-saclay.fr}}}
\date{2025}

\begin{document}
\maketitle
\begin{abstract}
Unbalanced optimal transport (UOT) is a natural extension of optimal transport (OT) allowing comparison between measures of different masses. It arises naturally in machine learning by offering a robustness against outliers. The aim of this work is to provide convergence rates of the regularized transport plans and potentials towards their original solution when both measures are weighted sums of Dirac masses. 
\end{abstract} 

\vskip\baselineskip\noindent


\tableofcontents

\section*{Notations}

The following notations we will be used throughout the article.
\begin{itemize}
\item \(\measspace^+(\mathcal{X})\) is the set of finite positive measures defined on \(\mathcal{X}\).
\item \(\inp{\cdot}{\cdot}\) is the inner product associated with the Euclidean norm \(\norm{\cdot}\) in \(\R^d\).
\item \(\preceq\) stands for a component-wise inequality between vectors in \(\R^{\mathcal{X}}\).
\item \(\lesssim\) stands for \(\leq\) up to a multiplicative universal positive constant. 
\item \(\Delta_d\) is the \((d\!-\!1)\)-dimensional simplex of \(\R^d\) defined by
    \[\Delta_d \defeq \set*{z \in (\R_+)^{d}}{\sum_{i=1}^d z_i = 1}.\]
\item \(\sqcup\) stands for a disjoint union of sets.
\item $\norm{\cdot}_A$ is the norm given by $\norm{x}_A = \sqrt{\inp{Ax}{x}}$ where $A$ is a positive definite self-adjoint operator.
\item $\domain f$ denotes the domain $\{x : f(x) < +\infty\}$ of a convex function valued in $\R\cup\{+\infty\}$.
\end{itemize}

\

\section{Introduction}

Given two positive measures \(\mu\) and \(\nu\) on compact domains \(\mathcal X \subseteq \R^d\) with same total mass and \(\mathcal Y\subseteq \R^d\) and a cost function \(c\!:\mathcal X\!\times\!\mathcal Y\to\R_+\), the Monge-Kantorovich optimal transport (OT) problem consists in
\begin{equation}
\label{intro:ot}
\inf_{\gamma \in \Gamma(\mu , \nu)} \int_{\mathcal{X}\times\mathcal{Y}}c(x,y)\dd\gamma(x,y) \end{equation}
where \(\Gamma(\mu,\nu)\) denotes the set of all probability measures on the product space \(\mathcal X\!\times\!\mathcal Y\) having \(\mu\) and \(\nu\) as marginals.
A popular way to solve numerically this problem consists in adding an entropy regularizing term, tuned via a non-negative parameter \(\varepsilon\), which leads to the following \emph{entropic optimal transport} problem
\begin{equation}
\label{intro:entropic}
\inf_{\gamma \in \Gamma(\mu , \nu)} \int_{\mathcal{X}\times\mathcal{Y}}c(x,y)\dd\gamma(x,y) + \varepsilon\entropy(\gamma \lvert \mu\! \otimes\! \nu), \end{equation}
where \(\entropy\) is the relative entropy.

Over the last decade this kind of approach has witnessed an impressive increasing interest since it has proved to be an efficient way to approximate OT problems. From the computational point of view, \labelcref{intro:entropic} can be solved by an iterative projections method, which turns out to correspond to the celebrated Sinkhorn's algorithm \cite{sinkhorn1967diagonal,sinkhornRelationshipArbitraryPositive1964}, successfully developed in the pioneering works \cite{cuturi2013sinkhorn,Galichon-Entropic,benamouIterativeBregmanProjections2015}. The simplicity of the implementation and the good convergence guarantees \cite{franklinScalingMultidimensionalMatrices1989, marinoOptimalTransportApproach2020} determined the success of entropic optimal transport for a wide range of applications, see for instance \cite{peyre2019computational} and the references therein.
Notice now if we look at \labelcref{intro:entropic} as a perturbation of \labelcref{intro:ot}, then it is quite natural to study the behavior of both optimal values and minimizers, when the regularization parameter varies, and in particular when \(\varepsilon\to 0\).
According to this let us mention some related works \cite{mikami1990variational,leonard2012schrodinger,leonardSurveySchrodingerProblem2014,carlierConvergenceEntropicSchemes2017}, which established \(\Gamma\)-convergence results of \labelcref{intro:entropic} to \labelcref{intro:ot}, \cite{pal2019difference,conforti2021formula,altschuler2022asymptotics,ecksteinConvergenceRatesRegularized2022,carlier2023convergence,malamut2025convergence,nenna2024convergence}, which provided asymptotic expansions for the cost. Here we want to focus on the convergence rate of the primal and dual variable in the spirit of \cite{cominetti1994asymptotic,weed2018explicit} in the discrete setting for a generalized OT problem: unbalanced OT.

By definition, classical optimal transport requires the two given measures to have the same total mass, which can be a limitation for certain practical applications, particularly in machine learning, or imaging. An extension of the classical framework, known as unbalanced optimal transport, addresses this limitation by relaxing the marginal constraints and incorporating a divergence term into the objective function to penalise the marginals. It was simultaneously introduced by \cite{chizat2018unbalanced,liero2018optimal,mons2016} and may be expressed in the following form, 
\begin{equation}\label{pb:uot}
 \boxed{\inf\limits_{\gamma \in \measspace^+(\mathcal{X}\times \mathcal{Y})} \int_{\mathcal{X}\times\mathcal{Y}}c(x,y)\dd\gamma(x,y) + \mathrm{D}_\phi(\gamma_1 \lvert \mu)+\mathrm{D}_\phi(\gamma_2 \lvert \nu)}, 
\end{equation}
where \(\gamma_i\) is the \(i\)-th marginal of the coupling \(\gamma\) and \(\mathrm{D}_\phi\) is a \(\phi\)-divergence, also known as Csiszàr divergence (see \Cref{Csiszar_divergence discrete}).

As in the usual optimal transport problem, one can introduce an entropy regularized version of the unbalanced optimal transport problem, depending on a parameter $\varepsilon >0$, as follows
\begin{equation}\label{pb:entro_uot}
 \boxed{\inf_{\gamma \in \measspace^+(\mathcal{X}\times \mathcal{Y})}  \int_{\mathcal{X}\times\mathcal{Y}}c(x,y)\dd\gamma(x,y) + \mathrm{D}_\phi(\gamma_1 \lvert \mu)+ \mathrm{D}_\phi(\gamma_2 \lvert \nu)+\varepsilon\entropy(\gamma|\mu\!\otimes\!\nu)},
\end{equation}

This paper focuses on this problem when both measures have finite support. The aim is to establish convergence rates of the regularized optimal variables, for the primal and dual problems, towards the original ones. The main contributions can be summarized in the following two theorems.
\begin{theoremA}[{Simplified statement of \Cref{main theorem}}]
For positive \(\varepsilon\), let \(\varepsilon\mapsto\gamma(\varepsilon)\) be the curve of optimal solutions to the regularized unbalanced optimal transport problem \labelcref{pb:entro_uot}.
Under the assumption that \(\mu\) and \(\nu\) are atomic measures, and under certain conditions on $\phi$, the trajectory \(\varepsilon\mapsto \gamma(\varepsilon)\) converges towards an optimizer \(\bar\gamma\) of \labelcref{pb:uot} at a rate of at least \(\sqrt{\varepsilon}\), i.e.
        \[\lVert \bar\gamma - \gamma(\varepsilon) \rVert \lesssim \sqrt{\varepsilon}.\]
\end{theoremA}

\begin{theoremB}[{Simplified statement of \Cref{main theorem 2})}]
For positive \(\varepsilon\), let  \(\varepsilon\mapsto\xi(\varepsilon)\) be the curve of optimal solutions to the regularized dual problem to \labelcref{pb:entro_uot}.
Under the assumption that \(\mu\) and \(\nu\) are atomic measures, and under certain conditions on $\phi$, the dual trajectory \(\varepsilon\mapsto \xi(\varepsilon)\) converges towards the dual optimizer \(\bar\xi\) at a rate of at least \(\varepsilon\), i.e.
        \[\lVert \bar\xi - \xi(\varepsilon) \rVert \lesssim \varepsilon.\]
\end{theoremB}

\paragraph{Outline of the paper.} 
The paper is organized as follows: in \Cref{sec:properties} we introduce the discrete balanced and unbalanced optimal transport problems, their dual problems, and their regularized counterparts. In particular we focus on the dual unbalanced problem and give some preliminary results. \Cref{sec:convergence} is devoted to the convergence of both the primal and the dual trajectory of solutions to the regularized unbalanced OT. In \Cref{sec:asymptotics} we provide the announced asymptotic rates of the regularized problem when the regularization parameter tends to \(0\). \Cref{sec:numerics} is devoted to numerical experiments and comparisons with the established asymptotics.

\section{Discrete unbalanced optimal transport }

\subsection{The setting}

Since we study the case of measures \(\mu,\nu\) with finite support, we assume from now on that \(\mathcal{X}, \mathcal{Y}\) are finite sets. Notice that measures \(\mu \in \measspace^+(\mathcal X)\),  \(\nu \in \measspace^+(\mathcal Y)\) and \(\gamma \in \measspace^+(\mathcal X\times \mathcal Y)\) can now be written as
\[\mu=\sum_{x\in\mathcal X}\mu_x\delta_x, \quad \nu=\sum_{y\in\mathcal Y}\nu_y\delta_y,\quad\text{and}\quad \gamma =\! \sum_{(x,y) \in \mathcal X \!\times\! \mathcal Y} \! \gamma_{x,y} \delta_{(x,y)},\]
and can thus be identified with their respective weight vectors \((\mu_x)_x\in \R^{\mathcal X}\), \((\nu_y)_y \in \R^{\mathcal Y}\) and \((\gamma_{x,y})_{(x,y)} \in \R^{\mathcal X\times \mathcal Y}\). In the same way, the cost \(c\) can be identified with the vector \((c(x,y))_{(x,y)} \in \R^{\mathcal X\times \mathcal Y}\). By introducing the linear map \(A\!: \R^{\mathcal{X}\times \mathcal{Y}} \to \R^{\mathcal{X} } \! \oplus \mathbb{R}^{\mathcal{Y}} \) defined by
\[A\gamma  = \left( \left(\sum_{y\in\mathcal{Y}} \gamma_{x,y} \right)_{x\in\mathcal{X}} , \left(\sum_{x\in\mathcal{X}} \gamma_{x,y} \right)_{y\in\mathcal{Y}} \right),\]
one can reformulate the marginal constraint as \(A\gamma = (\mu,\nu) \defeq q\) and the optimal transport problem and its regularized formulation can then be expressed as finite-dimensional problems as follows,
\begin{equation}
\label{OT_0}
    \inf_{\gamma \in(\R_+)^{\mathcal{X}\!\times\!\mathcal{Y}}}  \left\{ \inp{c}{\gamma} \: \vert \: A\gamma = q \right\} \tag{\(\text{OT}_0\)},
\end{equation}
and for any non-negative \(\varepsilon\),
\begin{equation}
\label{OT_epsilon}
    \inf_{\gamma \in \R^{\mathcal{X}\!\times\!\mathcal{Y}}} \left\{ \inp{c}{\gamma} + \varepsilon \entropy(\gamma) \: \vert \: A\gamma = q \right\}, \tag{\(\text{OT}_\varepsilon\)}
\end{equation}
where \(\inp{c}{\gamma} \defeq \sum_{(x,y)\in \mathcal X\!\times\!\mathcal Y}c(x,y)\gamma_{x,y}\) denotes the canonical inner product on \(\R^{\mathcal X\!\times\! \mathcal Y}\) and \(\entropy\!:\R^{\mathcal{X}\!\times\!\mathcal{Y}}\!\to\!\R\) is the discrete entropy, i.e. the entropy relative to the counting measure $\mathcal H^0$ defined by 
\[
\entropy(\gamma) \defeq \entropy(\gamma \vert \mathcal H^0 ) = \begin{dcases*}
\sum_{(x,y)\in\mathcal{X}\times\mathcal{Y}} \gamma_{x,y} (\log \gamma_{x,y} - 1)& if \(\gamma_{x,y} \geq 0\;\text{for every }(x,y)\)\\
 +\infty& otherwise.
\end{dcases*}
\]
We take the convention \(0 \log 0 = 0\) providing a continuous extension of \(t \mapsto t \log t\) on \(\R_+\).

\begin{remark}
Notice that, contrary to the problem stated in the general case in \labelcref{intro:entropic}, we took the entropy with respect to $\hdm^0$ instead of $\mu\otimes \nu$. However in the case where \(\mu\in\mathcal{P}(\mathcal{X)}\), \(\nu\in\mathcal{P}(\mathcal{Y)}\) and $A \gamma = (\mu,\nu)$, they satisfy the following relation:
\[\entropy(\gamma \:\rvert\:\mu\otimes \nu) = \entropy(\gamma)-\entropy(\mu) -\entropy(\nu)-2.\]
\end{remark}


\begin{remark}
Assuming that \(\mathcal{X}\) and \(\mathcal{Y}\) are finite subsets of \(\R^d\), say \(\mathcal{X}=\{x_1,\cdots,x_N\}\) and \(\mathcal{Y}=\{y_1,\cdots,y_M\}\), the linear map \(A\) can be associated with its matrix \(A_{mat}\) in the bases \( \mathcal{B}=(e_{x_1,y_1} , e_{x_1,y_2} , \cdots , e_{x_n,y_M}) \) and \( \mathcal{B}'=(e_{x_1} , \cdots , e_{x_n} , e_{y_1} , \cdots e_{y_M}) \) defined as
\[A_{mat} =
\begin{array}{cc} 
\begin{pmatrix}
\mathbb{1}_M & 0_M & \cdots & 0_M &\quad  \textnormal{I}_M \quad \\
0_M & \mathbb{1}_M & \cdots & 0_M & \textnormal{I}_M \\
\vdots & \vdots & \ddots & \vdots &\quad \vdots \quad \\
0_M & 0_M & \cdots & \mathbb{1}_M &\quad \textnormal{I}_M\quad  \\
\end{pmatrix}
\end{array}^T.\]
\end{remark}

\begin{proposition}
\label{link EOT dual}
For any \(\varepsilon>0\), \labelcref{OT_epsilon} and its dual admit a unique solution \(\gamma_\varepsilon\) and \(\xi_\varepsilon\) (up to the kernel of \(A^*\), which has dimension \(1\), for \(\xi_\varepsilon\)), respectively. Moreover,  the following relation holds
\[\gamma_\varepsilon = e^{\frac{1}{\varepsilon}(A^*\xi_\varepsilon - c )}.\]
\end{proposition}

We will not detail the proof since it is well known in the literature (for instance it is a by-product of the result in \cite[Proposition~4.4]{peyre2019computational}).

\indent The unbalanced optimal transport problem is a natural extension of classical optimal transport problem where one can consider measures of different masses by relaxing the mass conservation constraint; for more details we refer the reader to the seminal works \cite{chizat2018unbalanced,chizat2018interpolating,liero2018optimal,mons2016}. This approach, in particular, allows the creation or destruction of mass. To formalize this problem, we introduce a general function \(\F\!:\R^\mathcal{X}\!\oplus\R^\mathcal{Y}\!\to\!\R \cup\{+\infty\}\) which is used as a penalization on the marginals of the transport coupling. The unbalanced optimal transport problem then takes the form:
\begin{equation}
    \label{UOT}
    \boxed{\inf_{\gamma\in(\R_+)^{\mathcal{X}\!\times\!\mathcal{Y}}} \left\{ \inp{c}{\gamma} + \F(A\gamma)  \right \}}. \tag{\(\text{UOT}_0\)}
\end{equation}

In the same way as before, we consider a regularized version of this optimization problem which is given by the minimization of $\mathscr{C}_\varepsilon$ defined for every $\varepsilon >0$ by
\[\mathscr{C}_\varepsilon(\gamma) \defeq \inp{c}{\gamma} + \F(A\gamma) + \varepsilon\entropy(\gamma),\]
i.e. we consider
\begin{equation}
    \label{UOT_epsilon}
    \boxed{\inf_{\gamma\in \R^{\mathcal{X}\!\times\!\mathcal{Y}}} \left\{ \inp{c}{\gamma} + \F(A\gamma) + \varepsilon \entropy(\gamma)   \right \} }.\tag{\(\text{UOT}_\varepsilon\)}
\end{equation}

Throughout all this paper we will use the following assumptions for \(\F\).
\begin{enumerate}[(H$_1$),start=0]
    \item \label{H0}\(\F\) is a proper convex and lower semi-continuous function.
    \item[\mylabel{Hdom}{(H$_{\domain}$)}] There exists \((\mu,\nu) \in (\R_+^*)^{\mathcal X}\! \oplus (\R_+^*)^{\mathcal Y}\) such that  \(\F(\mu,\nu) < +\infty\).
\end{enumerate}

In the literature, for instance \cite{sejourne2022faster}, \(\F\) is often taken to be a Csiszàr \(\phi\)-divergence defined as follows.
\begin{definition}
\label{Csiszar_divergence discrete} (Discrete Csiszàr \(\phi\)-divergence) \\
Let \(q\in(\R_+)^N\) be a non-negative vector and \(\phi\!:\R \to \R_+ \cup \{+\infty\}\) an entropy function meaning that it is a convex and lower semi-continuous function such that $\domain \phi \subseteq \R_+$. The discrete Csiszàr \(\phi\)-divergence is defined on $\R^N$ by
\begin{equation*}
    \mathrm{D}_\phi(p \: | \: q) \defeq \sum_{i\,:\,q_i>0}q_i\phi\bigg(\frac{p_i}{q_i}\bigg)+\phi'_\infty\sum_{i\,:\,q_i=0}p_i,
\end{equation*}
where \(\phi_\infty'\defeq\lim_{t\to\infty}\phi(t)/t\) is the recession constant at infinity.
\end{definition}
\begin{remark}
The choice of considering a general function \(\F\) instead of a Csiszàr divergence is motivated by the fact that one may also want to consider  variational problems given by the sum of an optimal transport term and a congestion on the marginals, which are not Csiszàr divergences. All assumptions that we make on $\F$ can easily be translated as assumptions on the entropy function \(\phi\) when \(\F(\cdot)=\mathrm{D}_\phi(\cdot \vert q)\) and $(\mu, \nu) = q \succ 0$. In that case \myref{H0} and \myref{Hdom} are equivalent to the following:
\begin{enumerate}[(H$'_1$),start=0]
\item $\phi$ is a proper convex lower semi-continuous functions;
\item[\mylabel{Hdom'}{(H$'_{\domain}$)}] there exists $\alpha > 0$ such that $\phi(\alpha) < +\infty$.
\end{enumerate}
\end{remark}

\subsection{Properties of the regularized unbalanced optimal transport}\label{sec:properties}
Let us start by introducing the dual formulations of the unbalanced OT problem and its regularization. We will use the adjoint operator of \(A\) defined for any \(\xi=(\varphi , \psi ) \in \R^{\mathcal{X}}\!\oplus \R^{\mathcal{Y}} \) by
\[A^*\xi = \varphi \oplus \psi = (\varphi_x + \psi_y)_{(x,y)\in\mathcal{X}\!\times\!\mathcal{Y}}.\]
The dual problems of \labelcref{UOT} and \labelcref{UOT_epsilon}
are given by
    \begin{equation}
    \label{DUOT}
    \inf_{\xi\in \R^{\mathcal{X}}\!\oplus \R^{\mathcal{Y}}} \left\{ \F^*\!\left(-\xi \right) \: | \: A^*\xi \preceq c \right\}. \tag{\(\text{DUOT}_0\)}
    \end{equation}
and
\begin{equation}
\label{DUOT_epsilon}
\inf_{\xi \in \R^{\mathcal{X}} \oplus \R^{\mathcal{Y}}} \F^*(-\xi) + \sum_{(x,y) \in \mathcal{X} \times \mathcal{Y}} \varepsilon \exp\left( \frac{1}{\varepsilon} \inp{A^* \xi - c}{e_{x,y}} \right), \tag{\(\text{DUOT}_\varepsilon\)}
\end{equation}
where \( (e_{x,y})_{(x,y) \in \mathcal{X} \times \mathcal{Y}} \) denotes the canonical basis of \( \mathbb{R}^{\mathcal{X} \times \mathcal{Y}} \). We denote by \(\mathscr{K}_\varepsilon\) the functional of the dual problem, that is for $\xi = (\phi,\psi)$,
\[ \mathscr{K}_\varepsilon(\xi) = \F^*\!\left(-\xi\right) + \!\sum_{(x,y)\in\mathcal{X}\!\times\!\mathcal{Y}}\varepsilon e^{\frac{1}{\varepsilon}(\varphi_x +\psi_y-c(x,y))}.\]

\begin{remark}\label{legendre_transform_divergence}
    When \( \F(\cdot) = \mathrm{D}_\phi(\cdot \vert q) \), $\F^*$ is given by
    \[\F^*(\xi) = \inp{q}{\phi^*(\xi)} + I(\xi),\]
    where $I(\xi) = 0$ if $\xi_z \leq \phi'_\infty$ for every $z\in \mathcal X\sqcup \mathcal Y$ such that $q_z = 0$, and $I(\xi) = +\infty$ otherwise.
\end{remark}

We are now ready to state the basic results concerning existence, uniqueness and duality for these problems. 

\begin{proposition}
Let us assume \myref{H0} and \myref{Hdom}. Concerning the unbalanced OT problem, the following results hold:
\begin{enumerate}[(i)]
\item\label{strong_duality_unregularized} strong duality, i.e.
\[\inf \labelcref{UOT} + \inf \labelcref{DUOT} = 0;\]
\item\label{existence_primal_unregularized} existence for \labelcref{UOT} if $\F$ is coercive, and uniqueness up to $\ker A = \R (\one_{\R^{\mathcal X}},-\one_{\R^{\mathcal Y}})$ if $\F$ is strictly convex;
\item\label{existence_dual_unregularized} existence for \labelcref{DUOT}, and uniqueness if $\F^*$ is strictly convex.
\end{enumerate}
\end{proposition}

\begin{proof}

We begin by introducing the convex function \( G: \R^{\mathcal{X} \times \mathcal{Y}} \to \R \) defined by
\[
G(\gamma) = \inp{c}{\gamma} + \iota_{(\R_+)^{\mathcal{X} \times \mathcal{Y}}}(\gamma),
\]
where \( \iota_{(\R_+)^{\mathcal{X} \times \mathcal{Y}}} \) denotes the indicator function of the non-negative orthant in \( \R^{\mathcal{X} \times \mathcal{Y}} \). With this definition, the unbalanced optimal transport \labelcref{UOT} problem can be written as
\[
\inf_{\gamma \in \R^{\mathcal{X} \times \mathcal{Y}}} G(\gamma) + \F(A\gamma).
\]
By \myref{H0} and \myref{Hdom}, $\F$ is proper convex and lower semi-continuous, and there exists $(\mu,\nu) \in (\R_+^*)^{\mathcal X} \oplus (\R_+^*)^{\mathcal Y}$ such that $\F(\mu,\nu) < +\infty$. By definition of $A$, there exists $\gamma \in (\R_+^*)^{\mathcal X \times \mathcal Y}$ such that $A\gamma = (\mu,\nu)$. Since both $\F$ and $G$ are finite at $\gamma$ and $\F$ is continuous at $\gamma$, we may apply Fenchel–Rockafellar's Theorem \cite[Theorem~1.12]{brezisFunctionalAnalysisSobolev2011} to assert that strong duality holds and that there exists a solution to the dual problem:
\[
\inf_{\gamma \in \R^{\mathcal{X} \times \mathcal{Y}}} G(\gamma) + \F(A\gamma) = \max_{\xi \in \R^{\mathcal{X}} \oplus \R^{\mathcal{Y}}} -G^*(A^* \xi) - \F^*(-\xi),
\]
where \( \F^* \) and \( G^* \) are the Legendre–Fenchel transforms of \( \F \) and \( G \), respectively. Computing these transforms explicitly yields \labelcref{strong_duality_unregularized}.

Now, assume that \( \F \) is coercive. Notice that setting $\tilde A \gamma = (A\gamma, \inp{c}{\gamma})$ and $\tilde{\mathscr C}(p,t) = t + \F(p)$ for $(p,t) \in (\R^{\mathcal X} \oplus \R^{\mathcal Y}) \oplus \R$, \labelcref{UOT} can be equivalently written as
\[\inf_{\gamma \in (\R_+)^{\mathcal X \times \mathcal Y}} \tilde{\mathscr C}(\tilde A \gamma).\]
By \myref{H0}, it admits a competitor with finite cost. Let us notice that the cone $\Lambda = \tilde A \Bigl( (\R_+)^{\mathcal X \times \mathcal Y}\Bigr)$ is a closed convex cone, as a finitely generated convex cone (this can be shown by Caratheodory's theorem for cones). To get existence, we just have to show that $\tilde{\mathscr{C}}$ is coercive on $\Lambda$. Take a sequence $(A\gamma_n, \inp{c}{\gamma_n})_{n\in \N} \subseteq \Lambda$ which is unbounded. If $(A\gamma_n)$ is unbounded, $\F(A\gamma_n) \to +\infty$ by coercivity of $\F$, and since $\inp{c}{\gamma_n} \geq 0$ for every $n$, $\tilde{\mathscr C}(\gamma_n) \to +\infty$. If on the contrary $(A\gamma_n)$ is bounded, then $(\inp{c}{\gamma_n})$ is unbounded, and necessarily it tends to $+\infty$. Since $\F$ is convex proper and lower semi-continuous, it is bounded from below by a linear map, and since $(A\gamma_n)$ is bounded, then $(\F(\gamma_n))$ is lower bounded. This shows that $\tilde{\mathscr C}(\gamma_n) \to +\infty$ also in this case, hence $\tilde{\mathscr C}$ is coercive on the closed set $\Lambda$ thus the primal problem problem admits a minimizer. Existence for the dual problem has been justified above, and the uniqueness statements are straightforward consequences of strict convexity, thus \labelcref{existence_primal_unregularized} and \labelcref{existence_dual_unregularized} hold true.
\end{proof}

We shall need the definition of \emph{slack variable} associated with this problem later on.

\begin{definition}[Slack variable]
\label{I_0 and slack variable}
Assuming that \labelcref{DUOT} has a unique solution \(\bar \xi= (\bar\varphi,\bar\psi)\) then we denote by \(\kappa\) the slack variable, i.e. the gap between \(A^*\bar \xi\) and \(c\), given for every \((x,y) \in \mathcal{X}\!\times\!\mathcal{Y}\) by
\[\kappa_{x,y} \defeq c(x,y) - \bar\varphi_x - \bar\psi_y.\]
We denote by \(I_0\) the set of saturated constraint where $\kappa = 0$, namely
    \[I_0 \defeq \{ (x,y) \!\in\! \mathcal{X}\!\times\!\mathcal{Y} \:|\:  \bar\varphi_x + \bar\psi_y = c(x,y)\}.\]
\end{definition}

Let us now deal with existence and duality for the regularized problem.

\begin{proposition}
\label{prop: existence for regularized problems}
Under assumptions \myref{H0} and \myref{Hdom}, the following hold concerning the regularized unbalanced OT problem for every $\varepsilon >0$:
\begin{enumerate}[(i)]
\item\label{strong_duality_regularized} strong duality, i.e.
\[\inf \labelcref{UOT_epsilon} + \inf \labelcref{DUOT_epsilon} = 0;\]
\item\label{existence_primal_regularized} existence and uniqueness for \labelcref{UOT_epsilon};
\item\label{existence_dual_regularized} existence for \labelcref{DUOT_epsilon}, and uniqueness if $\F^*$ is strictly convex.
\end{enumerate}
\end{proposition}

\begin{proof}
Let \( \varepsilon > 0 \). Rather than using Fenchel-Rockafellar as in the unregularized case, we are going to show existence for the primal problem first then use subdifferential calculus to deduce strong duality and existence for the dual problem. Since \( \F \) is proper convex and lower semi-continuous, there exist \( a \in \R^{\mathcal{X}} \oplus \R^{\mathcal{Y}} \) and $b\in \R$ such that for any \( \xi \in \R^{\mathcal{X}} \oplus \R^{\mathcal{Y}} \), we have
\[
\F(\xi) \geq \inp{a}{\xi} + b.
\]
This implies that
\[
\mathscr{C}_\varepsilon(\gamma) = \inp{c}{\gamma} + \F(A\gamma) + \varepsilon \entropy(\gamma) \geq \inp{c + A^*a}{\gamma} + \varepsilon  \entropy(\gamma) + b.
\]
Using the superlinearity of the entropy functional \( \entropy \), we deduce that \( \mathscr{C}_\varepsilon \) is coercive. Since it is also lower semi-continuous, the \labelcref{UOT_epsilon} admits at least a solution. The strict convexity of \( \entropy \) ensures that \( \mathscr{C}_\varepsilon \) admits a unique solution, so that \labelcref{existence_primal_regularized} holds. Now, let \( \gamma_\varepsilon \) be the unique minimizer of \labelcref{UOT_epsilon}. Then, by first-order optimality conditions,
\[
0 \in \partial \mathscr{C}_\varepsilon(\gamma_\varepsilon).
\]
Define the functional \( G_\varepsilon : \R^{\mathcal{X} \times \mathcal{Y}} \to \R \) by
\[
G_\varepsilon(\gamma) \defeq \inp{c}{\gamma} + \varepsilon \entropy(\gamma)\quad\text{so that}\quad \mathscr{C}_\varepsilon = G_\varepsilon + \F.
\]
Then the optimality condition becomes
\[
0 \in A^* \, \partial \F(A\gamma_\varepsilon) + \partial G_\varepsilon(\gamma_\varepsilon),
\]
where we used assumption \myref{Hdom} (see \cite[Theorem 23.8]{rockafellar1970convex}) to guarantee the existence of a point $\gamma$ such that $G_\varepsilon$ is continuous at $\gamma$ and $A\gamma$ lies in the interior of $\domain \F$. Therefore, there exists \( u \in \partial \F(A\gamma_\varepsilon) \) and \( v \in \partial G_\varepsilon(\gamma_\varepsilon) \) such that
\[
0 = A^* u + v.
\]
By the Fenchel–Young equality, we have
\[
\F(A\gamma_\varepsilon) + \F^*(u) = \inp{u}{A\gamma_\varepsilon},
\quad \text{and} \quad
G_\varepsilon(\gamma_\varepsilon) + G_\varepsilon^*(v) = \inp{\gamma_\varepsilon}{v} = \inp{\gamma_\varepsilon}{\varepsilon -A^*u} = -\inp{u}{A\gamma_\varepsilon}.
\]
Summing the two equalities yields:
\[
\F(A\gamma_\varepsilon) + G_\varepsilon(\gamma_\varepsilon) + \F^*(u) + G_\varepsilon^*(-A^*u) = 0,
\]
from which we get 
\[
\F(A\gamma_\varepsilon) + G_\varepsilon(\gamma_\varepsilon) = -\F^*(u) - G_\varepsilon^*(-A^*u).
\]
This shows at the same time that $u$ is a solution to the dual problem and strong duality holds. Since the Legendre transforms $\F^*, G^*_\varepsilon$ can be computed explicitly we have \labelcref{strong_duality_regularized}. Obviously, uniqueness holds if $\F^*$ is strictly convex, and \labelcref{existence_dual_regularized} holds.
\end{proof}

\begin{remark}
Uniqueness of the solution for \labelcref{DUOT_epsilon} holds modulo \(\ker A^*\), even in the absence of strict convexity of 
\(\F^*\).
\end{remark}

\begin{remark}
\label{q>0}
In the case where \( \F(\cdot) = \mathrm{D}_\phi(\cdot \vert q) \) and $q \succ 0$, we have already seen that \myref{Hdom} holds if and only if $\phi(\alpha) < +\infty$ for some $\alpha > 0$, in which case the previous proposition holds, and we have existence for \labelcref{DUOT_epsilon}. Let us show that existence for \labelcref{DUOT_epsilon} may fail without the assumption \(q\succ 0\). Consider a function $\phi$ such that $\phi'_\infty = +\infty$ and assume for example that \(q_{x_0}=\mu_{x_0} =0\) for some \(x_0 \in \mathcal{X}\). By contradiction assume that there a solution \(\xi(\varepsilon)\) of \labelcref{DUOT_epsilon}. We define a sequence \((\xi^n(\varepsilon))_{n\in\mathbb{N}}\) such that for any \(n\in\mathbb{N}\), 
\[\xi_{z}^n(\varepsilon)=
\begin{dcases*}
  -n & if  \(z=x_0\) \\
  \xi_z(\varepsilon) & otherwise.
\end{dcases*}
\]
Since $\phi'_\infty = +\infty$, by \Cref{legendre_transform_divergence} we have
\[\F^*(-\xi^n) = \inp{q}{\phi^*(-\xi^n(\varepsilon))} = \F^*(-\xi)\]
and
\begin{align*}
     \mathscr{K}_\varepsilon(\xi^n(\varepsilon)) &= \F^*(-\xi^n(\varepsilon)) + \sum_{(x,y)\in\mathcal{X}\!\times\!\mathcal{Y}} \varepsilon e^{\frac{1}{\varepsilon}\inp{A^*\xi^n(\varepsilon) -c} = {e_{x,y}}} \\
     &= \mathscr{K}_\varepsilon(\xi(\varepsilon)) + \varepsilon\sum_{y\in\mathcal{Y}} e^{\frac{1}{\varepsilon}(\psi_y(\varepsilon) - c(x_0 , y))- \frac{n}{\varepsilon}} - e^{\frac{1}{\varepsilon}(\varphi_{x_0}(\varepsilon) + \psi_y(\varepsilon) - c(x_0 , y))} \\
     &\underset{n \to +\infty}{\longrightarrow}\mathscr{K}_\varepsilon(\xi(\varepsilon)) - \varepsilon\sum_{y\in\mathcal{Y}}e^{\frac{1}{\varepsilon}(\varphi_{x_0}(\varepsilon) + \psi_y(\varepsilon) - c(x_0 , y))} < \mathscr{K}_\varepsilon(\xi(\varepsilon)),
\end{align*}
which is a contradiction.
\end{remark}

\begin{proposition}
\label{prop:primal_dual}
    Let $\varepsilon > 0$. Under \myref{H0} and \myref{Hdom}, the optimal regularized transport coupling for \labelcref{UOT_epsilon} denoted \(\gamma(\varepsilon)\) and any solution of \labelcref{DUOT_epsilon} denoted \(\xi_\eps\) are related through the following formula
    \[ \gamma(\varepsilon) = \exp(\tfrac{1}{\varepsilon}(A^*\xi_\varepsilon-c)).\]
\end{proposition}

\begin{proof}
We observe that by strong duality, as stated in \Cref{prop: existence for regularized problems},
\[\F(\gamma(\varepsilon)) + \inp{c}{\gamma(\varepsilon)} + \varepsilon \entropy(\gamma) +\F^*(-\xi_\varepsilon) - \varepsilon \sum_{(x,y)\in\mathcal X\times \mathcal Y} \exp\left(\frac 1\varepsilon \inp{A^*\xi_\varepsilon-c}{e_{x,y}}\right) = 0.
\]
Denoting $(\gamma_{1,\varepsilon},\gamma_{2,\varepsilon}) = A \gamma(\varepsilon)$ and $(\varphi_\varepsilon,\psi_\varepsilon) = \xi_\varepsilon$, by Fenchel-Young inequality $\F(\gamma(\varepsilon)) + \F^*(-\xi_\varepsilon) \geq \inp{\gamma(\varepsilon)}{\xi_\varepsilon} = \inp{\gamma_{1,\varepsilon}}{\varphi_\varepsilon}+\inp{\gamma_{2,\varepsilon}}{\psi_\varepsilon}$ thus
\[\inp{c}{\gamma(\varepsilon)} + \varepsilon\entropy(\gamma(\varepsilon))\leq \inp{\gamma_{1,\varepsilon}}{\varphi_\varepsilon}+\inp{\gamma_{2,\varepsilon}}{\psi_\varepsilon} + \varepsilon \sum_{(x,y)\in\mathcal X\times \mathcal Y} \exp\left(\frac 1\varepsilon \inp{A^*\xi-c}{e_{x,y}}\right).\]
The left-hand side is the primal functional of \labelcref{OT_epsilon} with fixed marginals $(\gamma_{1,\varepsilon},\gamma_{2,\varepsilon})$ evaluated at $\gamma(\varepsilon)$, and the right-hand side is its dual functional evaluated at $\xi_\varepsilon$. Since the former is always greater than the latter, we have equality and $\gamma(\epsilon)$ and $\xi_\varepsilon$ are respectively optimal for the primal and dual regularized OT problems with these fixed marginals. It classically follows that 
\[
\gamma(\varepsilon) = \exp\big( \tfrac{1}{\varepsilon}(A^*\xi_\varepsilon - c) \big).
\]
\end{proof}

\section{Convergence and regularity of the trajectories}\label{sec:convergence}

This section is devoted to the study of the trajectory in \(\varepsilon\) of the solutions for the primal and dual problems, \labelcref{UOT_epsilon} and \labelcref{DUOT_epsilon} respectively. Before making more assumptions on the function \(\F\), let us recall some definitions for real valued functions defined on a normed vector space \(E\).

\begin{definition}
\label{def: superlinear} A function \(f\!:E \to \R \cup \{+\infty\} \) is \emph{super-linear at infinity} if 
    \[\lim\limits_{\| x\|\to +\infty} \frac{f(x)}{\| x\|} =+\infty.\]
\end{definition}

\begin{definition}
    \label{weakly concave}
    A proper convex and lower semi-continuous function \(f\!:E\to \R\cup \{+\infty\}\) is \emph{strongly convex} at \(x_0\in \domain f\), if there exists \(r>0\), \(x_0^*\in \partial f(x_0)\) and \(\lambda(x_0,r) > 0\) such that
    \[\forall x \in B(x_0,r) \cap \domain f, \quad f(x)\geq f(x_0) + \inp{x_0^*}{x-x_0} + \frac{\lambda(x_0,r)}{2}\lVert x-x_0\rVert^2.\]
\end{definition}

Let us introduce a further series of hypotheses that we shall use several times throughout the paper.

\begin{enumerate}[(H$_1$),start=1]
    \item \label{H1}\(W \cap (\R_+)^{\mathcal{X\!\times\mathcal{Y}}} \subseteq \domain \F\), for some neighborhood \(W\) of the origin;
    \item \label{H2}\(\F\) is super-linear at infinity;
    \item \label{H3} \(\F^*\) is strongly convex at every point of its domain;
    \item \label{H4} \(\F^*\) is \(\mathcal{C}^2\) on \(\interior (\domain \F^*)\).
\end{enumerate}

Notice that \myref{H1} implies \myref{Hdom} and \myref{H3} implies strict convexity of $\F^*$.

\begin{remark}
In the case where \( \F(\cdot) = \mathrm{D}_\phi(\cdot \vert q) \), with \( q \succ 0\) fixed, all these assumptions on $\F$ are equivalent to the following assumptions on $\phi$:
\begin{enumerate}[(H$'_1$),start=1]
\item $\domain \phi$ contains some neighborhood of $0$ in $\R_+$;
\item $\phi'_\infty = +\infty$;
\item $\phi^*$ is strongly convex at every point of its domain;
\item $\phi^*$ is $\mathcal{C}^2$ on  \(\interior (\domain \phi^*)\).
\end{enumerate}
This can be proven straighforwardly using the expression of $\F^*$ given in \Cref{legendre_transform_divergence}.
\end{remark}

The following proposition shows in particular that under \myref{H2}, \(\F^*\) has full domain. Its proof can be found for example in \cite{santambrogio2023course}, but we provide it for the sake of completeness.
\begin{proposition}
\label{domain F^*}
Under assumptions \myref{H0} and \myref{H2}, the Legendre transform \(\F^*\) is bounded on every ball.
\end{proposition}

\begin{proof}
Let \(R>0\). Recall that for every \(\xi\in B(0,R)\) we have
\[\F^*(\xi) \defeq \sup_{v \in \R^\mathcal{X}\!\oplus\R^\mathcal{Y}} \inp{v}{\xi} - \F(v).\]
By super-linearity at infinity of \(\F\), there exists \(M>0\) such that \(\F(v) \leq R\|v\|\) whenever \(\|v\| \geq M \), in which case \(\|v\| \| \xi \| - \F(v) \leq 0\), while for \(v \leq M\) we have
\[
\inp{v}{\xi} -\F(v) \leq M \|\xi\| - \inf_{\overline{B}(0, M)} \F,
\]
the infimum being a finite quantity since a proper lower semi-continuous convex function always has an affine minorant. We have shown that \(\F^*\) is bounded from above on every ball, and since it is also proper convex and lower semi-continuous by \myref{H0}, it is also bounded from below on every ball.
\end{proof}



\begin{lemma}[Uniform coercivity of \(\mathscr K_\varepsilon\)]\label{uniform_coercivity_kantorovich_functional}
    Under assumptions \myref{H0} and \myref{H1}, the functional \(\inf\limits_{\varepsilon>0} \mathscr{K}_\varepsilon\) is coercive.
\end{lemma}

\begin{proof}
For any \(\xi=(\varphi,\psi)\in \R^\mathcal{X}\!\oplus\R^\mathcal{Y}\), one has
\begin{align*}
\inf_{\varepsilon>0} \mathscr{K}_\varepsilon(\xi) &= \F^*\left(-\xi\right) + \inf_{\varepsilon>0}\!\sum_{(x,y)\in\mathcal{X}\!\times\!\mathcal{Y}}\!\varepsilon e^{\frac{1}{\varepsilon}(\varphi_x +\psi_y-c(x,y))} \\
&\geq \F^*\left(-\xi\right) + \inf_{\varepsilon>0}\!\sum_{(x,y)\in \mathcal{X}\!\times\!\mathcal{Y}}\varepsilon e^{\frac{1}{\varepsilon}\left(\varphi_x +\psi_y-c(x,y)\right)_+} \\
&\geq \F^*\left(-\xi\right) + \sum_{(x,y)\in \mathcal{X}\!\times\!\mathcal{Y}}\left( \varphi_x +\psi_y-c(x,y)\right)_+\defeq V(\xi),
\end{align*}
where $u_+ = \max\{0,u\}$ denotes the positive part of $u$. Let us show that \(V\) is coercive. Let \((\xi^n)_{n \in \mathbb{N}}=((\varphi^n , \psi^n))_{n \in \mathbb{N}}\) be a sequence in \(\R^{\mathcal{X}} \!\oplus \R^{\mathcal{Y}}\) such that \(\| \xi^n \| \to +\infty\) as \(n \to +\infty\). On one hand, if there exists \( z_0 \in \mathcal{X} \sqcup \mathcal{Y} \) such that, after taking a subsequence, \(\xi_{z_0}^n \to -\infty\) then we have, using \myref{H1}, that there exists \(r \in\R^*_+ \) such that \(r e_{z_0}\in\domain \F \). It leads to 
\[V(\xi^n) \geq \F^*(-\xi^n) \geq -r\xi_{z_0}^n - \F(r e_{z_0}) \xrightarrow{n\to+\infty} +\infty.\]
On the other hand, if that is not the case, then for any \(z \in \mathcal{X}\sqcup \mathcal{Y}\),  \(\xi_{z}^n\) is bounded from below by some constant $M <0$. Thus there exists \(z_0 \in \mathcal{X}\sqcup \mathcal{Y} \) such that \(\xi_{z_0}^n \to +\infty\). Without loss of generality we assume that \(z_0 \in \mathcal{X}\). Then we have, using \myref{H1} again, that for any \(y\in \mathcal{Y}\), there exists \(r \in\R^*_+ \) such that \(r e_y\in\domain \F \). For $n$ large enough, it leads to 
\begin{align*}
V(\xi^n) &\geq  \F^*(-\xi^n) + (\varphi_{z_0}^n + \psi_y^n - c(z_0,y))_+  \\
&\geq  -r\psi_y^n - \F(r e_y) + \varphi_{z_0}^n + \psi_y^n - c(z_0,y) \\
&\geq \abs{1-r}M - F(re_y) + \varphi_{z_0}^n - c(z_0,y)\underset{n \to +\infty}{\longrightarrow}  +\infty.
\end{align*}
We have shown that \(V(\xi^n) \to +\infty\) along a subsequence, but since the same reasoning applies to any arbitrary subsequence taken in the beginning, we get the result for the whole sequence, which establishes the coercivity of \(V\).
\end{proof}

\subsection{Dual trajectory}

In this section we show that the primal trajectory is regular, that it sastisfies an ODE (that will be used in \Cref{sec:asymptotics}), and that it converges as $\varepsilon \to 0$.

\begin{proposition}
    \label{regularity of dual trajectory}
     Assume that assumptions \myref{H0}-\myref{H4} hold. Then, there exists a unique solution $\xi(\varepsilon)$ to \labelcref{DUOT_epsilon} for every \(\varepsilon > 0\) and the trajectory \(\varepsilon \mapsto \xi(\varepsilon)\) is a \(\mathcal{C}^1\) function on \(\R_+^*\).
\end{proposition}

\begin{proof}
The existence and uniqueness of the solution to \labelcref{DUOT_epsilon} comes from \Cref{prop: existence for regularized problems} and the strict convexity of the functional \(\mathscr{K}_\varepsilon\), resulting from \myref{H3}.

Let us first consider any \(\varepsilon_0\in \R_+^*\) and denote by \(f\!: \R_+^*\times (\R^\mathcal{X}\!\oplus\R^\mathcal{Y})\to\R\) the function defined by \(f(\varepsilon,\xi)=D\!\mathscr{K}_\varepsilon(\xi)\). It is well-defined and of class $\mathcal{C}^2$ by \Cref{domain F^*} and \myref{H4}. By optimality of $\xi(\varepsilon_0)$, \(f(\varepsilon_0,\xi(\varepsilon_0))=0\) and by \myref{H3}, we get \(\det D_\xi f(\varepsilon_0,\xi(\varepsilon_0)) = \det D^2\! \mathscr{K}_{\varepsilon_0} (\xi(\varepsilon_0)) > 0\). Thus the implicit function theorem allows to deduce that the trajectory is of class \(\mathcal{C}^1\) on \(\R_+^*\).
\end{proof}

In the following $\xi(\varepsilon)$ will always denote the unique solution of \labelcref{DUOT_epsilon}.

\begin{proposition}\label{proposition_ODE}
Assume that assumptions \myref{H0}-\myref{H4} hold. Then the trajectory \(\varepsilon\mapsto \xi(\varepsilon)\) satisfies the following ordinary differential equation (ODE) on \(\R_+^*\),
\begin{equation*}
    D^2\!\mathscr{K}_\varepsilon(\xi(\varepsilon)) \dot{\xi}(\varepsilon) + \left(\frac{\partial}{\partial \varepsilon} D\!\mathscr{K}_\varepsilon\right)(\xi(\varepsilon)) = 0,
\end{equation*}
which can be explicitly written, denoting $\gamma(\varepsilon)$ the unique solution to \labelcref{UOT_epsilon}, as
\begin{equation}
    \label{ODE}
    A \diagonal \gamma(\varepsilon) A^*\dot{\xi}(\varepsilon) + \varepsilon D^2\!\F^*\!(-\xi(\varepsilon))\dot{\xi}(\varepsilon) -  A \diagonal \gamma(\varepsilon)\frac{A^*\xi(\varepsilon)-c}{\varepsilon} = 0.
    \tag{\(\text{ODE}\)}
\end{equation}
Here $\diagonal \gamma$ denotes the linear operator on $\R^{\mathcal X\times \mathcal Y}$ given by $(\diagonal \gamma) \rho = (\gamma_{x,y} \rho_{x,y})_{x,y}$. 
\end{proposition}

\begin{proof}
For every \(\varepsilon > 0\), by using the optimality of \(\xi(\varepsilon)\) for the functional \(\mathscr{K}_\varepsilon\), one can obtain
\[D\!\mathscr{K}_\varepsilon(\xi(\varepsilon)) = 0.\]
The function \(\varepsilon \mapsto D\!\mathscr{K}_\varepsilon(\xi(\varepsilon))\) is \(\mathcal{C}^1\) on \(\R_+^*\). Differentiating the previous equality with respect to $t$ and using the chain rule leads to
\[D^2\!\mathscr{K}_\varepsilon(\xi(\varepsilon)) \dot{\xi}(\varepsilon) + \left(\frac{\partial}{\partial \varepsilon} D\!\mathscr{K}_\varepsilon\right)(\xi(\varepsilon)) = 0.\]
Recalling the expression of $\mathscr{K}_\varepsilon$, i.e. $\mathscr{K}_\varepsilon(\gamma) = \F^*(-\xi) + \sum_{(x,y) \in \mathcal{X} \times \mathcal{Y}} \varepsilon \exp\left( \frac{1}{\varepsilon} \inp{A^* \xi - c}{e_{x,y}} \right)$ and the relation between $\gamma(\varepsilon)$ and $\xi(\varepsilon)$ given in \Cref{prop:primal_dual}, we may compute \(D\!\mathscr{K}_\varepsilon\), \(D^2\!\mathscr{K}_\varepsilon\), \(D\!\mathscr{K}_\varepsilon\) and $(\partial/\partial\varepsilon)(D\mathscr{K}_\varepsilon)$ explicitly, so as to obtain \cref{ODE}.
\end{proof}

\begin{proposition}\label{convergence_dual_optimizer}
    Under assumptions \myref{H0}-\myref{H3}, the trajectory \(\varepsilon\mapsto \xi(\varepsilon)\) converges as \(\varepsilon\rightarrow0\) towards the unique solution \(\bar\xi\) to \cref{DUOT}.
\end{proposition}

\begin{proof} Recall that $\bar\xi$ is unique by strict convexity of $\F^*$. Using the optimality of \(\xi(\varepsilon)\) for \labelcref{DUOT_epsilon}, we leverage the following inequality :
\[
     \mathscr K_\varepsilon(\xi(\varepsilon)) = \F^*\!\left(-\xi(\varepsilon)\right) + \!\sum_{(x,y)\in \mathcal{X}\!\times\!\mathcal{Y}}\!\varepsilon e^{\frac{1}{\varepsilon}\inp{A^*\xi(\varepsilon)-c}{e_{x,y}}} \leq  \F^*\!\left(-\bar\xi\right) + \!\sum_{(x,y)\in \mathcal{X}\!\times\!\mathcal{Y}}\!\varepsilon e^{\frac{1}{\varepsilon} \inp{A^*\bar\xi-c}{e_{x,y}}}.\]
The constraint \(A^*\bar\xi \preceq c\) allows to deduce that
\begin{equation}
    \label{ineq_proof}
 \mathscr K_\varepsilon(\xi(\varepsilon)) =\F^*\!\left(-\xi(\varepsilon)\right) + \!\sum_{(x,y)\in \mathcal{X}\!\times\!\mathcal{Y}}\! \varepsilon e^{\frac{1}{\varepsilon}\inp{A^*\xi(\varepsilon)-c}{e_{x,y}}} \leq \F^*\!\left(-\bar\xi\right) + \varepsilon \abs{\mathcal X}\abs{\mathcal Y}.
\end{equation}
From this, and the uniform coercivity of $\mathscr K_\varepsilon$ as stated in \Cref{uniform_coercivity_kantorovich_functional} (which applies thanks to \myref{H0} and \myref{H2}), we deduce that the trajectory $\varepsilon \mapsto \xi(\varepsilon)$ is bounded on \((0,1)\). 
By Bolzano-Weierstrass Theorem, we may consider a sequence \( (\varepsilon_n)_{n \in \mathbb{N}} \) such that \( \xi(\varepsilon_n) \to \tilde{\xi} \) as \( n \to +\infty \), where \( \tilde{\xi} \) is an accumulation point of the trajectory (at \(0\)). Since \((\xi(\varepsilon_n))_{n\in \N}\) is bounded, by \Cref{domain F^*} we know that \(\abs{\F^*(-\xi(\varepsilon_n))}\) is bounded by some constant \(C >0\) for any \(n\in\N\). Thus by \labelcref{ineq_proof}, for every \( n \) and every \( (x, y) \in \mathcal{X} \times \mathcal{Y} \),
\[
 \varepsilon_n e^{\frac{1}{\varepsilon_n}\inp{A^*\xi(\varepsilon_n)-c}{e_{x,y}}} \leq C+\F^*\!\left(-\bar\xi\right)+ \varepsilon_n \abs{\mathcal X}\abs{\mathcal Y}.
\]
If \(\alpha \defeq  \inp{A^* \tilde\xi-c}{e_{x,y}}  > 0\) then for \(n\) large enough we would have \(\varepsilon_n e^{\frac{1}{\varepsilon_n}\inp{A^*\xi(\varepsilon_n)-c}{e_{x,y}}} \geq \varepsilon_n e^{\frac{\alpha}{2\varepsilon_n}} \to +\infty\): a contradiction with the boundedness of the right-hand side of the previous equation. We conclude that \( A^*\tilde{\xi} \preceq c \), meaning that any accumulation point is feasible for the non-regularized dual problem. Moreover by \labelcref{ineq_proof} we also have
\[
\F^*\!\left(-\xi(\varepsilon_n)\right) \leq \F^*\!\left(-\bar\xi\right) + \varepsilon_n \abs{\mathcal X}\abs{\mathcal Y},
\]
and taking the limit as \( n \to +\infty \), we obtain by lower semi-continuity of $\F^*$:
\[
\F^*\!(-\tilde{\xi}) \leq \F^*\!(-\bar\xi).
\]
This shows that any accumulation point of the trajectory is optimal for \labelcref{DUOT} thus equal to \(\bar\xi\), hence \(\xi(\varepsilon)\) converges to \( \bar\xi \).
\end{proof}

\subsection{Primal trajectory}

The purpose of the following proposition is to show that the curve \(\varepsilon \mapsto \gamma(\varepsilon )\) converges to \(\bar\gamma\), the solution of \cref{UOT} of minimal entropy (which is unique by strict convexity of \(\entropy\)), under suitable assumptions.

\begin{proposition}\label{convergence_primal_optimizer}
    Under assumptions  \myref{H0}-\myref{H4}, the trajectory \(\varepsilon\mapsto \gamma(\varepsilon)\) converges towards \(\bar\gamma\), the solution of \cref{UOT} with minimal entropy \(\entropy\), as \(\varepsilon\to 0\). Moreover, \(\varepsilon\mapsto \gamma(\varepsilon)\) is a \(\mathcal{C}^1\) function on \(\mathbb{R}_+^*\).
\end{proposition}

\begin{proof} We start by showing  that the trajectory is bounded. For any positive \(\varepsilon\), by optimality of \(\gamma(\varepsilon)\) for \labelcref{UOT_epsilon} we have that 
\begin{equation}
\label{ineq full}
\inp{c}{\gamma(\varepsilon)} + \F(A\gamma(\varepsilon)) + \varepsilon \entropy(\gamma(\varepsilon)) \leq \inp{c}{\bar\gamma} + \F(A\bar\gamma) +  \varepsilon \entropy(\bar\gamma).
\end{equation}
Since \(\bar\gamma\) is optimal for \labelcref{UOT}, we have
\[
\inp{c}{\bar\gamma} + \F(A\bar\gamma) \leq \inp{c}{\gamma(\varepsilon)} + \F(A\gamma(\varepsilon))
\]
and combining it with \labelcref{ineq full} we get
\begin{equation}
    \label{inequality on the entropy}
    \entropy(\gamma(\varepsilon)) \leq \entropy(\bar\gamma).
\end{equation}
By coercivity of \(\entropy\), we can conclude that the trajectory \((\gamma(\varepsilon))_{\varepsilon> 0}\) is bounded. It implies that there exists \((\varepsilon_n)_{n \in \mathbb{N}}\) converging to $0$ such that \(\gamma(\varepsilon_n) \to \tilde{\gamma}\) as \(n \to +\infty\), where \(\tilde{\gamma}\) is an accumulation point at $0$. Since $\entropy$ is lower bounded, $\entropy(\gamma(\varepsilon))$ is bounded, thus taking \(\varepsilon =\varepsilon_n\) in \cref{ineq full} and passing to the limit shows that \(\tilde{\gamma}\) is a solution of \labelcref{UOT}. Finally, using \Cref{inequality on the entropy}, taking \(\varepsilon=\varepsilon_n\) and passing to the limit, by lower semicontinuity of $\entropy$ we obtain
\[
\entropy(\tilde{\gamma}) \leq \entropy(\bar\gamma),
\]
which implies that \(\tilde{\gamma} = \bar\gamma\), and thus \(\gamma(\varepsilon) \to \bar\gamma\) as \(\varepsilon \to 0\).
Finally, the regularity of the trajectory is a direct consequence of \Cref{prop:primal_dual} and \Cref{regularity of dual trajectory}.
\end{proof}



\section{Asymptotics for the regularized unbalanced problem}\label{sec:asymptotics}
This section is devoted to an asymptotic analysis of the regularized unbalanced problem as \(\varepsilon\to 0\). We divide our study into two parts: we start with the dual problem and then we tackle the primal problem. The proofs are strongly inspired from \cite{cominetti1994asymptotic}, which dealt with linear problems with exponential penalization (corresponding to entropy penalization in the primal). We adapt their strategy to the non-linearity induced by the function \(\F\).

\subsection{Asymptotic analysis of the dual trajectory}
We start by introducing the linear space \(E_0 \defeq \spn\{A e_{x,y} \: \vert \:(x,y)\in I_0 \}\), where \(I_0\) is the set of saturated constraint as set in \Cref{I_0 and slack variable}. It corresponds to the space of marginals of (signed) measures concentrated on the set of saturated constraints.

\begin{lemma}
\label{DF*_in_E0}
    Under assumptions \myref{H0} - \myref{H4}, the quantity \(D\!\F^*(-\bar\xi)\) belongs to \(E_0\).
\end{lemma}

\begin{proof}
Consider the curve \( \varepsilon \mapsto \xi(\varepsilon) \) of solutions to \labelcref{DUOT_epsilon}. By \myref{H0}, \myref{H2}, \myref{H4} and \Cref{domain F^*}, \(\F^*\) is differentiable on \(\R^{\mathcal X} \oplus \R^{\mathcal Y}\), thus the first order optimality condition reads as:
\begin{align*} 
D\!\mathscr{K}_\varepsilon(\xi(\varepsilon)) &= 0 \\
\Longleftrightarrow
D\!\F^*\!\left(-\xi(\varepsilon)\right) + \sum_{(x,y)\in \mathcal{X}\!\times\!\mathcal{Y}}\gamma_{x,y}(\varepsilon) A e_{x,y} &= 0.
\end{align*}
By \Cref{convergence_dual_optimizer} and \Cref{convergence_primal_optimizer}, \(\xi(\varepsilon) \to \bar\xi\) and \(\gamma(\varepsilon) \to \bar\gamma\) as \(\varepsilon\to 0\), thus taking the limit in the above expression leads to
\[
D\!\F^*(-\bar\xi) + \sum_{(x,y)\in I_0}\bar\gamma_{x,y} A e_{x,y} =0,
\]
 where we have used the fact that for any \((x,y)\notin I_0\), one has
\[\gamma_{x,y}(\varepsilon)=e^{\frac{1}{\varepsilon}\inp{A^*\xi(\varepsilon)-c}{e_{x,y}}} \underset{\varepsilon \to 0}{\longrightarrow} 0.\]
From this, we conclude that \(D\!\F^*(-\bar\xi)\in E_0\).
\end{proof}

For any positive \(\varepsilon\), we consider the quantity 
\[d(\varepsilon)\defeq \tfrac{1}{\varepsilon}(\xi(\varepsilon)-\bar\xi)\] and we look at \(\mathscr{K}_\varepsilon\) as a function of \(d\) instead of \(\xi\), namely we introduce the new   functional \(\tilde{\mathscr{K}}_\varepsilon\!: \R^\mathcal{X}\!\oplus \R^\mathcal{Y} \to \R\) defined as
\[\tilde{\mathscr{K}}_\varepsilon(d) \defeq \F^*\left(-\varepsilon d - \bar\xi\right) + \!\sum_{(x,y)\in \mathcal{X}\!\times \!\mathcal{Y}}\!\varepsilon e^{\inp{A^* d - \frac{1}{\varepsilon} \kappa}{e_{x,y}}}, \]
where \(\kappa\) is the gap between \(A\bar\xi\) and \(c\) as defined in \Cref{I_0 and slack variable}. Notice that it is defined in such a way that \(\tilde{\mathscr{K}}_\varepsilon(\frac{1}{\varepsilon}(\xi-\bar\xi)) = \mathscr K_\varepsilon(\xi)\).
\begin{proposition}\label{boundedness_d}
     Under assumptions  \myref{H0}-\myref{H4}, the curve \(\varepsilon \mapsto d(\varepsilon)\) is bounded for \(\varepsilon\) sufficiently small.
\end{proposition}

\begin{proof} Let \(\varepsilon>0\). Notice that for any \((x,y) \in I_0\), one has
\[ \inp{A^*d(\varepsilon)}{e_{x,y}} = \frac{1}{\varepsilon}\inp{A^*\xi(\varepsilon)-A^*\bar\xi}{e_{x,y}}  = \frac{1}{\varepsilon}\inp{A^*\xi(\varepsilon) - c}{e_{x,y}},\]
and thus by \Cref{convergence_primal_optimizer},
\[\inp{A^*d(\varepsilon)}{e_{x,y}} = \log \gamma_{x,y}(\varepsilon) \xrightarrow[\varepsilon\to 0]{} \log \bar\gamma_{x,y}.\]
Since $\bar\gamma$ has finite entropy and $\bar\gamma \succ 0$, we deduce from this and from the regularity of \(\xi(\varepsilon)\) (stated in \Cref{regularity of dual trajectory}), that \( \inp{A^*d(\varepsilon)}{e_{x,y}} \) is a bounded quantity on \((0,1)\) for any \((x,y)\in I_0\), which in turn implies that \(d^{(0)}(\varepsilon)\), the projection of \(d(\varepsilon)\) onto \(E_0\), is bounded. Besides, since \(\tilde{\mathscr{K}}_\varepsilon(\frac{1}{\varepsilon}(\xi-\bar\xi)) = \mathscr K_\varepsilon(\xi)\) and \(\xi(\varepsilon)\) is the unique minimizer of \(\mathscr K_\varepsilon\), \(d(\varepsilon)\) is the unique minimizer of \(\tilde{\mathscr K}_\varepsilon\). Comparing it with \(d^{(0)}(\varepsilon)\) yields
\[
\tilde{\mathscr{K}}_\varepsilon(d(\varepsilon)) \leq \tilde{\mathscr{K}}_\varepsilon(d^{(0)}(\varepsilon)),
\]
which implies
\begin{multline*}
    \F^*\left(-\varepsilon d(\varepsilon) - \bar\xi\right) + \sum_{(x,y)\in I_0}\!\varepsilon e^{\inp{A^*d(\varepsilon) -\frac{1}{\varepsilon}\kappa}{e_{x,y}}}\\
    \leq \F^*\left(-\varepsilon d^{(0)}(\varepsilon) - \bar\xi\right) + \sum_{(x,y)\in \mathcal X \times \mathcal Y} \varepsilon e^{\inp{A^*d^{(0)}(\varepsilon) -\frac{1}{\varepsilon}\kappa}{e_{x,y}}}.
\end{multline*}
Using that \(\inp{A^*d(\varepsilon)}{e_{x,y}} = \inp{d(\varepsilon)}{A e_{x,y}} = \inp{A^*d^{(0)}(\varepsilon)}{e_{x,y}}\) for any \((x,y)\in I_0\), we obtain
\begin{equation}\label{eq_rate_F*_1}
    \F^*\!\left(-\varepsilon d(\varepsilon) - \bar\xi\!\right) \leq  \F^*\!\left(-\varepsilon d^{(0)}(\varepsilon) - \bar\xi\!\right) + \!\sum_{(x,y)\notin I_0}\!\varepsilon e^{\inp{A^*d^{(0)}(\varepsilon)- \frac{1}{\varepsilon}\kappa}{e_{x,y}}}.
\end{equation}
Notice that \(\varepsilon d(\varepsilon) = \xi(\varepsilon) - \bar\xi \to 0\) and \(-\bar\xi \in \domain \F^*\), because $\F^*$ has full domain thanks to \Cref{domain F^*}. Thus, by \myref{H3}, for some \(r \in (0,1)\), for all $\varepsilon \leq r$ we have
\begin{equation}\label{eq_rate_F*_2}
    \F^*\left(-\varepsilon d(\varepsilon) - \bar\xi\right) \geq \F^*\left( - \bar\xi\right) - \varepsilon \inp{D\!\F^*\left(-\bar\xi\right)}{d(\varepsilon)} + \varepsilon^2\frac{\lambda_0(-\bar\xi,r)}{2}  \norm{d(\varepsilon)}^2,
\end{equation}
where $\lambda_0(-\bar\xi,r) > 0$, and by \myref{H4} and Taylor's expansion, we have
\begin{equation}\label{eq_rate_F*_3}
    \F^*\left(-\varepsilon d^{(0)}(\varepsilon) - \bar\xi\right) \leq \F^*\left( -\bar\xi\right) -\varepsilon \inp{D\F^*\left(-\bar\xi\right)}{d^{(0)}(\varepsilon)} + C \varepsilon^2
\end{equation}
for some \(C > 0\) and \(\varepsilon\) sufficiently small. By \Cref{DF*_in_E0}, we have 
\[ \inp{D\!\F^*\!\left( -\bar\xi \right)}{d(\varepsilon)} = \inp{D\!\F^*\!\left( -\bar\xi \right)}{d^{(0)}(\varepsilon)},\]
and combining this with \labelcref{eq_rate_F*_1}, \labelcref{eq_rate_F*_2} and \labelcref{eq_rate_F*_3}, we obtain
\[
       \lambda_0 \norm{d(\varepsilon)}^2 \leq  \sum_{(x,y)\notin I_0} \frac{2}{\varepsilon} e^{\inp{A^*d^{(0)}(\varepsilon) - \frac{1}{\varepsilon}\kappa}{e_{x,y}}} + C.
\]
We have seen that \(d^{(0)}(\varepsilon)\) is bounded, and since \(\kappa_{x,y} > 0\) for every \((x,y) \not\in I_0\), the right-hand side is bounded, thus \(d(\varepsilon)\) is bounded as well for \(  \varepsilon \leq r \).
\end{proof}

\begin{proposition}\label{convergence_d}
Under assumptions \myref{H0}-\myref{H4}, the trajectory \(\varepsilon\mapsto d(\varepsilon)\) converges towards a finite quantity denoted \(\bar d\).
\end{proposition}

Before going into the proof, we introduce the following useful lemma.

\begin{lemma}
    \label{set of solutions}
Let \(E'\) be a normed vector space, \(f\!: E' \to\R\) a strictly convex function and \(L\!:E\to E' \) a linear operator. If the set of solutions \(\mathcal{U}\) of the problem 
\[\inf_{x\in E} f(Lx)\]
is nonempty, then for any \(x_0 \in \mathcal{U}\), one has \[\mathcal{U} = x_0 + \ker L.\]
\end{lemma}

\begin{proof}
Let \(x_0 \in \mathcal{U}\). If \(x_1 = x_0 + y\), with \(y \in  \ker L\), then $Lx_0 = Lx_1$ thus \(f(Lx_1) = f(Lx_0)\), and it follows that \(x_1 \in \mathcal{U}\).

Conversely, let \(x_1 \in \mathcal{U}\). We set $y = x_1-x_0$ and $x_2 = \frac 12(x_0+x_1) = x_0 + \frac 12 y$. Assuming that \(y \notin \ker L\), \(L x_1 \neq Lx_2\) and by strict convexity of \(f\) we have  
\[
f(Lx_0) = \frac{1}{2} f(Lx_0) + \frac{1}{2} f(Lx_1) > f(Lx_2),
\]
where \(x_2 \defeq (x_0+x_1)/2 = x_0 + \frac{1}{2} y\). This contradicts the optimality of \(x_0\), therefore \(y \in \ker L\) and $x_1 \in x_0 + \ker L$.
\end{proof} 

\begin{proof}[Proof of \Cref{convergence_d}]
Let \( \bar d \) be an accumulation point at \(0\) of the curve \(\varepsilon\mapsto d(\varepsilon)\), meaning that \( d(\varepsilon_n) \to \bar d \) as \( n \to +\infty \) for some sequence \(\varepsilon_n \to 0\), which exists thanks to \Cref{boundedness_d}. Since \( \xi(\varepsilon_n) \) is a minimizer of \labelcref{DUOT_epsilon} for $\varepsilon=\varepsilon_n$, we obtain the following equation:
\[
D\!\F^*\!\left(-\varepsilon_n d(\varepsilon_n)-\bar\xi\right) + \sum_{(x,y)\in \mathcal{X}\!\times\!\mathcal{Y}}e^{\inp{A^*d(\varepsilon_n) - \frac{1}{\varepsilon_n} \kappa}{e_{x,y}}}A e_{x,y} = 0.
\]
Taking the limit in this expression leads to
\[
D\!\F^*(-\bar\xi) + \sum_{(x,y)\in I_0}e^{\inp{A^*\bar d}{e_{x,y}}}A e_{x,y}=0.
\]
We recognize that is the first order optimality condition of the following differentiable convex function,
\[
\mathrm{G} : z\in\R^\mathcal{X}\!\oplus \R^\mathcal{Y} \mapsto \inp{D\!\F^*(-\bar\xi)}{z} +  \sum_{(x,y)\in I_0}e^{\inp{A^*z}{e_{x,y}}},
\]
so that \(\bar d\) is a minimizer of $\mathrm{G}$. Notice that by \Cref{DF*_in_E0}, \(\inp{D\!\F^*(-\bar\xi)}{z}\) is a linear function of the quantities \(\inp{A^* z}{e_{x,y}}\) for $(x,y) \in I_0$, and since the exponential is strictly convex,  we may apply \Cref{set of solutions} to assert that the set of optimal solutions \( \mathcal{U} \) of \(\mathrm{G}\) is given by
\[
\mathcal{U}= \bar d + \bigcap\limits_{(x,y)\in I_0} \ker Ae_{x,y}. 
\]
Now, let us introduce \( \bar d_1\in \mathcal{U} \) such that \( \bar d_1 \neq \bar d \). We define \( d_1(\varepsilon) \) by
\[
d_1(\varepsilon) \defeq d(\varepsilon) + \bar d_1 - \bar d
\]
which satisfies \( d_1(\varepsilon_n) \rightarrow \bar d_1 \) as \( n\rightarrow +\infty \). 

We know that \( d(\varepsilon) \) is a minimizer solution of the functional \( \tilde{\mathscr K}_\varepsilon \), thus
\[\tilde{\mathscr{K}}_\varepsilon(d(\varepsilon_n)) \leq \tilde{\mathscr{K}}_\varepsilon(d_1(\varepsilon_n)),\]
or equivalently
\begin{multline*}
    \F^*\left(-\varepsilon_n d(\varepsilon_n) - \bar\xi\right) + \sum_{(x,y)\in \mathcal X \times \mathcal Y}\!\varepsilon_n e^{\inp{A^*d(\varepsilon_n)}{e_{x,y}}}\\
    \leq \F^*\left(-\varepsilon d_1(\varepsilon_n) - \bar\xi\right) + \sum_{(x,y)\in \mathcal X \times \mathcal Y}\varepsilon_n e^{\inp{A^*d_1(\varepsilon_n)}{e_{x,y}}}.
\end{multline*}
Since \(d(\varepsilon_n)\) and \(d_1(\varepsilon_n)\) are bounded, we can write the second order Taylor expansion of \(\F^*\) at \(-\bar\xi\) to get
\begin{multline*}
    \frac{\norm{d(\varepsilon_n)}_{D^2 \F^*(-\bar\xi)}^2}2 + \sum_{(x,y)\in \mathcal X \times \mathcal Y}\!\frac{1}{\varepsilon_n}e^{\inp{A^*d(\varepsilon_n) -\frac{1}{\varepsilon_n}\kappa}{e_{x,y}}}\\
    \leq \frac{\norm{d_1(\varepsilon_n)}_{D^2 \F^*(-\bar\xi)}^2}2 + \sum_{(x,y)\in \mathcal X \times \mathcal Y}\frac{1}{\varepsilon_n} e^{\inp{A^*d_1(\varepsilon_n) -\frac{1}{\varepsilon_n}\kappa}{e_{x,y}}} + o(1),
\end{multline*}
where we have denoted $\norm{u}_A \defeq \inp{Au}u$ and omitted the first-order terms, which cancelled thanks to \Cref{DF*_in_E0}. Besides, thanks to the fact that \(d(\varepsilon_n)-d_1(\varepsilon_n) \in \bigcap_{(x,y)\in I_0} \ker Ae_{x,y}\), the terms corresponding to \((x,y) \in I_0\) in the sum are equal on both sides. By non-negativity of the exponential, we forget the terms corresponding to \((x,y) \not\in I_0\) in the left-hand side, and we get:

\[
    \| d(\varepsilon_n) \|_{D^2\!\F^*\!(-\bar\xi)}^2 \leq  \| d_1(\varepsilon_n) \|_{D^2\!\F^*\!(-\bar\xi)}^2 + \!\sum\limits_{(x,y)\notin I_0}\!2\varepsilon_n e^{\inp{Ad_1(\varepsilon_n) - \frac{1}{\varepsilon_n}\kappa}{e_{x,y}}} + o(1).
\]
Passing to the limit, we obtain
\[
\| \bar d \|_{D^2\!\F^*\!(-\bar\xi)}^2 \leq \| \bar d_1 \|_{D^2\!\F^*\!(-\bar\xi)}^2.
\]
Thus, we have shown that \( \bar d \) is the unique solution of the strictly convex (by \myref{H3}) problem:
\[
\inf\limits_{z\in \mathcal{U}}\: \| z \|_{D^2\!\F^*\!(-\bar\xi)}^2.
\]
We have shown that $d(\varepsilon)$ has a unique accumulation point at $0$, and since it is bounded on a neighborhood of $0$, it converges to $\bar d$.
\end{proof}

Now, we shall study the derivative of \(d(\varepsilon)\) with respect to \(\varepsilon\) which is equal to \(\dot{d}(\varepsilon)= -\frac{1}{\varepsilon} d(\varepsilon) +\frac{1}{\varepsilon} \dot{\xi}(\varepsilon)\). 

\begin{theorem}
\label{main theorem 2}
Under assumptions \myref{H0}-\myref{H4}, the convergence rate of the curve \( \varepsilon \mapsto \xi(\varepsilon) \) is at least of order \( O(\varepsilon) \), i.e. for any \(\varepsilon\) sufficiently close to \(0\), one has 
\[
\| \xi(\varepsilon) - \bar\xi \|  \lesssim \varepsilon.
\]
\end{theorem}

\begin{proof} Let \( \varepsilon > 0 \). We define a functional \( \Lambda_\varepsilon \) by
\[
\Lambda_\varepsilon(\sigma) \defeq \norm*{A^*\sigma +\tfrac{c - A^*\bar\xi}{\varepsilon^2}}_{\diagonal(\gamma(\varepsilon))}^2 + \varepsilon \norm*{ \sigma + \frac{d(\varepsilon)}{\varepsilon}}_{D^2 \!\F^*\!(-\xi(\varepsilon))}^2.
\]
Note that \( \dot{d}(\varepsilon) \) is optimal for \( \Lambda_\varepsilon \), since the latter is convex and 
\begin{align*}
    D\Lambda_\varepsilon(\sigma) &= 2A\diagonal(\gamma(\varepsilon))\left(A^*\sigma + \frac{c - A^*\bar\xi}{\varepsilon^2}\right) + 2 \varepsilon D^2\!\F^*\!(-\xi(\varepsilon))\left(\sigma + \frac{d(\varepsilon)}{\varepsilon}\right)
\shortintertext{and thus, using the ODE satisfied by $\varepsilon\mapsto \xi(\varepsilon)$ and established in \Cref{proposition_ODE},}
    D\Lambda_\varepsilon(\dot{d}(\varepsilon)) &= \frac{2}{\varepsilon}\left(A\diagonal(\gamma(\varepsilon))A^*\dot{\xi}(\varepsilon) + \varepsilon D^2\!\F^*\!(-\xi(\varepsilon))\dot{\xi}(\varepsilon) - A \diagonal(\gamma(\varepsilon))\frac{A^*\xi(\varepsilon)-c}{\varepsilon}\right)=0.
\end{align*}
Using this property, we get the following,
\begin{align*}
 & &
    \Lambda_\varepsilon(\dot{d}(\varepsilon)) &\leq \Lambda_\varepsilon(0) \\
    &\iff& \norm*{A^*\frac{\dot\xi(\varepsilon)}{\varepsilon}}_{\diagonal(\gamma(\varepsilon))}^2 + \varepsilon \norm*{ \frac{\dot\xi(\varepsilon)}{\varepsilon}}_{D^2 \!\F^*\!(-\xi(\varepsilon))}^2 &\leq \norm*{\tfrac{c - A^*\bar\xi}{\varepsilon^2}}_{\diagonal(\gamma(\varepsilon))}^2 + \varepsilon \norm*{ \frac{d(\varepsilon)}{\varepsilon}}_{D^2 \!\F^*\!(-\xi(\varepsilon))}^2.
\end{align*}
This in turn implies
\begin{align*}
    & &
\varepsilon \norm*{\frac{\dot{\xi}(\varepsilon)}{\varepsilon}}_{D^2\!\F^*\!(-\xi(\varepsilon))}^2 &\leq \sum_{(x,y)\notin I_0}\tfrac{1}{\varepsilon^2} \gamma_{x,y}(\varepsilon)\inp{c - A^*\bar\xi}{e_{x,y}}^2 + \varepsilon \norm*{ \frac{\bar\xi - \xi(\varepsilon)}{\varepsilon^2}}_{D^2\!\F^*\!(-\xi(\varepsilon))}^2 \\
    &\implies&
    \| \dot{\xi}(\varepsilon) \|_{D^2\!\F^*\!(-\xi(\varepsilon))}^2 &\leq \sum_{(x,y)\notin I_0} \tfrac{1}{\varepsilon} \gamma_{x,y}(\varepsilon)\inp{c - A^*\bar\xi}{e_{x,y}}^2 + \tfrac{1}{\varepsilon^2} \| \bar\xi - \xi(\varepsilon)  \|_{D^2\!\F^*\!(-\xi(\varepsilon))}^2 \\
&\implies& 
\| \dot{\xi}(\varepsilon) \|^2 &\leq \frac{C_1}{\varepsilon \lambda_{\min}(\varepsilon)} \max_{(x,y)\notin I_0}\gamma_{x,y}(\varepsilon) + \frac{\lambda_{\max}(\varepsilon)}{\lambda_{\min}(\varepsilon)} \| d(\varepsilon) \|^2,
\end{align*}
where \( C_1 \defeq \sum_{(x,y)\notin I_0} \kappa_{x,y}^2 \) and \(\lambda_{\min}(\varepsilon)>0\) and \(\lambda_{\max}(\varepsilon)<+\infty\) are respectively the smallest and the largest eigenvalues of \(D^2\!\F^*\!(-\xi(\varepsilon))\). We have already seen that the trajectory \( \varepsilon \mapsto d(\varepsilon) \) is bounded for \( \varepsilon \leq r\) for some small $r>0$ thus there exists \( C_2 > 0 \) such that for $\varepsilon \leq r$, $\norm{d(\varepsilon)} \leq C_2$, and for any \( (x,y) \notin I_0 \)
\[
\gamma_{x,y}(\varepsilon) = e^{\frac{1}{\varepsilon}\inp{A^*\xi(\varepsilon) - c}{e_{x,y}}} = e^{\inp{A^*d(\varepsilon)}{e_{x,y}}} e^{-\frac{1}{\varepsilon}\kappa_{x,y}} \leq C_2 e^{-\frac{1}{\varepsilon}\kappa^*},
\]
where \(\kappa^*\defeq\min\limits_{(x,y)\notin I_0} \kappa_{x,y}\). One obtains  
\[
\| \dot{\xi}(\varepsilon) \|^2 \leq  \tfrac{C}{\varepsilon} e^{-\frac{1}{\varepsilon}\kappa^*} +  K,
\]
where  
\[
C \defeq \sup_{0<\varepsilon \leq r} \frac{C_1 C_2}{\lambda_{\min}(\varepsilon)}<+\infty \quad \text{and} \quad K \defeq \sup_{0<\varepsilon \leq r} C_2^2 \frac{\lambda_{\max}(\varepsilon)}{\lambda_{\min}(\varepsilon)}  <+\infty,
\]
these quantities being finite because $\F^*$ is of class $\mathcal C^2$ by \myref{H4}. Moreover, \(\varepsilon \mapsto \tfrac{1}{\varepsilon} e^{-\frac{1}{\varepsilon}\kappa^*}\) converges exponentially fast towards \(0\) and thus \(\tfrac{1}{\varepsilon} e^{-\frac{1}{\varepsilon}\kappa^*} \lesssim 1\) when \(\varepsilon\) is sufficiently close to \(0\). Consequently, we have as \(\varepsilon\) goes to \(0\)
\[
\| \dot{\xi}(\varepsilon) \|  \lesssim  1,
\]
which leads to
\[
    \| \bar\xi - \xi(\varepsilon) \|  \leq \int_0^{\varepsilon} \| \dot{\xi}(\omega)\| \,\dd\omega \lesssim \int_0^{\varepsilon}  \,\dd\omega \lesssim \varepsilon . 
\]
\end{proof}

\subsection{Asymptotic analysis of the primal trajectory}
\begin{theorem}
\label{main theorem}
Under assumptions \myref{H0}-\myref{H4}, the convergence rate of the curve \(\varepsilon \mapsto \gamma(\varepsilon)\) is at least of order \(O(\sqrt{\varepsilon})\), i.e. for any \(\varepsilon\) close to \(0\), one has
\[\| \gamma(\varepsilon) - \bar\gamma\| \lesssim \sqrt{\varepsilon}. \]
\end{theorem}

\begin{proof} Let \(\varepsilon>0\) sufficiently small in order to have \(\| \xi(\varepsilon) -\bar\xi\| \lesssim  \varepsilon \) as shown in \Cref{main theorem 2}. We have seen in \Cref{proposition_ODE} that $\dot\xi$ solves \labelcref{ODE}, which can be recognized as the first-order optimality condition of the following convex optimization problem
\[\inf_{\R^\mathcal{X}\!\oplus \R^\mathcal{Y}} \Psi_\varepsilon(\xi) \defeq \inf\limits_{\xi\in \R^\mathcal{X}\!\oplus \R^\mathcal{Y}} \| A^*\xi -\log\gamma(\varepsilon)\|^2_{\diagonal(\gamma(t))}  + \varepsilon \| \xi\|^2_{ D^2\!\F^*\!(-\xi(\varepsilon))}.\]
We compare $\dot\xi(t)$ with the competitor \(h(\varepsilon)\defeq \tfrac{1}{\varepsilon}(\xi(\varepsilon) - \bar\xi)\) to obtain
\begin{align*}
    \Psi_\varepsilon(\dot{\xi}(\varepsilon)) &\leq \Psi_\varepsilon(h(\varepsilon))\\
     &\leq \sum_{(x,y)\in \mathcal{X}\times\mathcal{Y}}  \gamma_{x,y}(\varepsilon)\left(\inp{A^*h(\varepsilon)}{e_{x,y}}  - \log\gamma_{x,y}(\varepsilon)\right)^2 + \varepsilon \| h(\varepsilon)\|^2_{D^2\mathrm{F}^*(-\xi(\varepsilon))} \\   
   &= \sum_{(x,y)\in \mathcal{X}\times\mathcal{Y}} \tfrac{1}{\varepsilon^2}\gamma_{x,y}(\varepsilon)\inp{A^*\xi(\varepsilon)- A^*\bar\xi -A^*\xi(\varepsilon)+c}{e_{x,y}}^2 + \tfrac{1}{\varepsilon}\| \xi(\varepsilon)-\bar\xi \|^2_{D^2\mathrm{F}^*(-\xi(\varepsilon))} \\
    &\leq \sum_{(x,y)\notin I_0} \tfrac{1}{\varepsilon^2} \gamma_{x,y}(\varepsilon)\inp{A^*\bar\xi-c}{e_{x,y}}^2 + \tfrac{1}{\varepsilon} \lambda_{\max}(\varepsilon)\|  \xi(\varepsilon)-\bar\xi\|^2 \\
   &\lesssim \frac{1}{\varepsilon^2}\max\limits_{(x,y) \notin I_0} \gamma_{x,y}(\varepsilon) + \varepsilon ,
\end{align*}
where the last inequality is obtained using that \(\sum_{(x,y)\notin I_0} \inp{A^*\bar\xi -c}{e_{x,y}}^2\) and \(\lambda_{\max}(\varepsilon)\), the largest eigenvalue of \( D^2\mathrm{F}^*\!(-\xi(\varepsilon))\), are bounded quantities. From this, for any \((x,y) \in \mathcal{X}\!\times\!\mathcal{Y}\), one can obtain 
\[ \gamma_{x,y}(\varepsilon)\inp{A^*\dot{\xi}(\varepsilon) - \log \gamma(\varepsilon)}{e_{x,y}}^2 \lesssim \frac{1}{\varepsilon^2} \max_{(x,y)\notin I_0} \gamma_{x,y}(\varepsilon) + \varepsilon.\] 
Moreover, we compute
\begin{align*}
\frac{\textit{d}}{\textit{d}\varepsilon}\left(\varepsilon\log \gamma(\varepsilon)\right) &= \log \gamma(\varepsilon) + \varepsilon \frac{\dot{\gamma}(\varepsilon)}{\gamma(\varepsilon)} \\
&= \frac{1}{\varepsilon}( A^*\xi(\varepsilon)-c)  + \varepsilon\frac{(\tfrac{1}{\varepsilon}A^*\dot{\xi}(\varepsilon) -\tfrac{1}{\varepsilon^2} (A^*\xi(\varepsilon)-c))\gamma(\varepsilon)}{\gamma(\varepsilon)}\\
&= A^*\dot{\xi}(\varepsilon),
\end{align*}
and in particular \( A^*\dot{\xi}(t) - \log\gamma(\varepsilon)=\varepsilon \frac{\dot{\gamma}(\varepsilon)}{\gamma(\varepsilon)}\). Thus, we have
\begin{align*}
   & &
   \gamma_{x,y}(\varepsilon)\left(\varepsilon \frac{\dot{\gamma}_{x,y}(\varepsilon)}{\gamma_{x,y}(\varepsilon)}\right)^2 &\lesssim \frac{1}{\varepsilon^2}\max_{(x,y)\notin I_0} \gamma_{x,y}(t) + \varepsilon \\
    &\iff&
    \varepsilon^2 \frac{\dot{\gamma}_{x,y}(\varepsilon)^2}{\gamma_{x,y}(t)} &\lesssim \frac{1}{\varepsilon^2} \max_{(x,y)\notin I_0}\gamma_{x,y}(\varepsilon) + \frac{1}{\varepsilon}  \\
    &\implies&
    \dot{\gamma}_{x,y}(\varepsilon)^2 &\lesssim \frac{1}{\varepsilon^4} \max_{(x,y)\notin I_0}\gamma_{x,y}(\varepsilon) + \frac{1}{\varepsilon},
\end{align*}
where the last inequality uses that \(\varepsilon \mapsto \gamma(\varepsilon)\) is bounded when \(\varepsilon\) is sufficiently small. Since \(d(\varepsilon)\) is bounded for
\(\varepsilon\) small enough and recalling that the transport coupling can be expressed in terms of the dual variable as 
\[\gamma(\varepsilon)= e^{\tfrac{1}{\varepsilon}(A^*\xi(\varepsilon)  - c)} = e^{\tfrac{1}{\varepsilon} A^*(\xi(\varepsilon)-\bar\xi)- \tfrac{1}{\varepsilon}\kappa} = e^{ A^*d(\varepsilon) - \tfrac{1}{\varepsilon}\kappa},\]
it follows that there exists \(C>0\) such that
\[\forall (x,y) \notin I_0, \quad \gamma_{x,y}(\varepsilon) \leq Ce^{-\tfrac{1}{\varepsilon}\kappa^*}, \]
with \(\kappa^*=\min_{(x,y)\notin I_0} \kappa_{x,y}\).
It allows to conclude that for any \((x,y) \notin I_0\), \(\gamma_{x,y}(\varepsilon)\) converges exponentially fast towards \(0\). Consequently, we obtain the asymptotic estimate
\[\abs{\dot{\gamma}_{x,y}(\varepsilon)} \lesssim \frac{1}{\sqrt{\varepsilon}}. \]
Finally, we have
\begin{align*}
\| \gamma(\varepsilon) - \bar\gamma \| &\leq \int_{0}^{\varepsilon} \|\dot{\gamma}(\omega) \| \dd\omega \lesssim \int_{0}^{\varepsilon} \frac{1}{\sqrt{\omega}} \,\dd\omega \lesssim \sqrt{\varepsilon}.
\end{align*}
which allows to obtain 
\[ \| \gamma(\varepsilon) - \bar\gamma \| \lesssim \sqrt{\varepsilon}.\qedhere\]
\end{proof}

\section{Numerical experiments}\label{sec:numerics}

We consider two finite sets \(\mathcal{X}, \mathcal{Y} \subseteq \R^d\) with cardinality \(N\) and \(M\), respectively and  divergences which take the form \(\mathrm{F}(\cdot) = \mathrm{D}_\phi(\cdot \vert q)\) where \(q\) represents the marginals and \(\phi\) is an entropy function. \\

For each divergence \(\mathrm{F}\) considered in this section, we provide numerical illustrations of the expected upper bounds on the convergence rates for both the primal and the dual problem.
We consider two different datasets. 
The first dataset is composed of two discretized Gaussian measures on a regular 1-dimensional grid, also with unequal masses: \(|\mu| = 11\) and \(|\nu| = 10\). 
The second and third datasets consist of two point clouds \(\mathcal{X}\) and \(\mathcal{Y}\) with uniform weights, respectively of dimension \(2\) and \(3\). Notably, \(\mathcal{Y}\) contains outliers points that are significantly distant from the others. The measures supported on these sets have different total masses, with \(|\mu| = 11\) and \(|\nu| = 10\).

The convergence rate plots are presented on a logarithmic scale and are compared to their theoretically expected rates.

\begin{figure}[H]
    \centering
    \begin{subfigure}{0.33\textwidth}
        \centering
        \includegraphics[width=\linewidth]{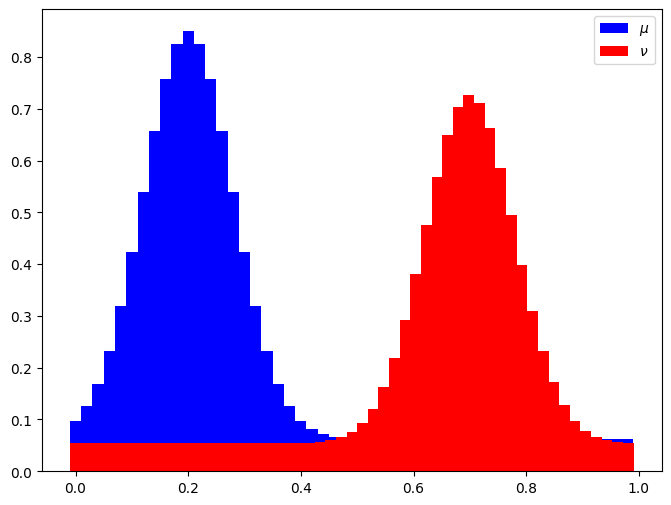}
    \end{subfigure}
    \hfill
    \begin{subfigure}{0.33\textwidth}
        \centering
        \includegraphics[width=\linewidth]{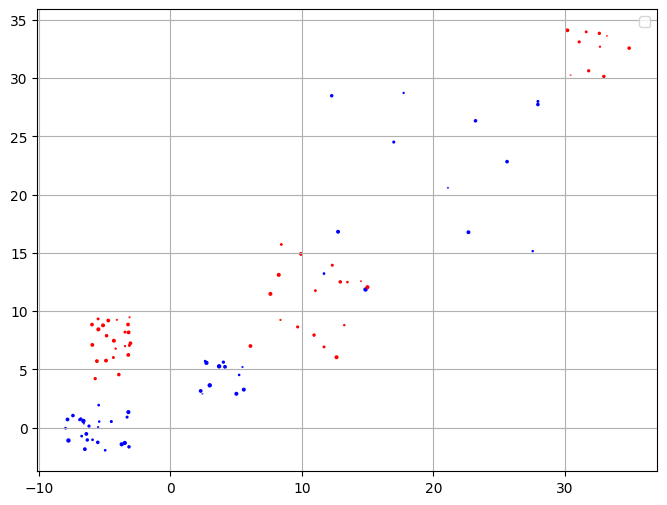}
    \end{subfigure}
    \hfill
    \begin{subfigure}{0.3\textwidth}
        \centering
        \includegraphics[width=\linewidth]{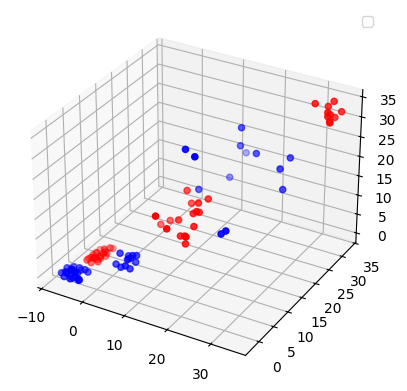}
    \end{subfigure}
    \caption{Three Datasets. The left panel shows two discretized Gaussian measures. The center panel shows two weighted point clouds in 2D. The right panel shows two weighted point clouds in 3D.}
\end{figure}

\subsection{Kullback-Leibler divergence}
The most widely used case corresponds to the function \(\phi\) defined as \(\phi(x) \defeq x(\log x - 1)\) on \(\R_+\) and \(+\infty\) otherwise. It gives the so-called Kullback-Leibler divergence, defined for any \(\psi \in \R^{N\!+\!M}\) by
\[
\mathrm{F}(\psi) = \begin{dcases*} \sum_{i=1}^{N\!+\!M} \psi_i \left( \log\frac{\psi_i}{q_i} - 1 \right) & if \(\psi_i \geq 0\; (\forall i)\) \\
+\infty & otherwise.
\end{dcases*}
\]
The Legendre transform of \(\phi\) is given for any \(y \in \R\) by
\[
\phi^*(y) = e^{y}.
\]
It follows that \(\phi^*\) is a \(\mathcal{C}^2\) function on \(\R_+^*\) that is strongly convex at every point of its domain.

\begin{figure}[H]
    \centering
    \begin{subfigure}{0.4\textwidth}
        \centering
        \includegraphics[width=\linewidth]{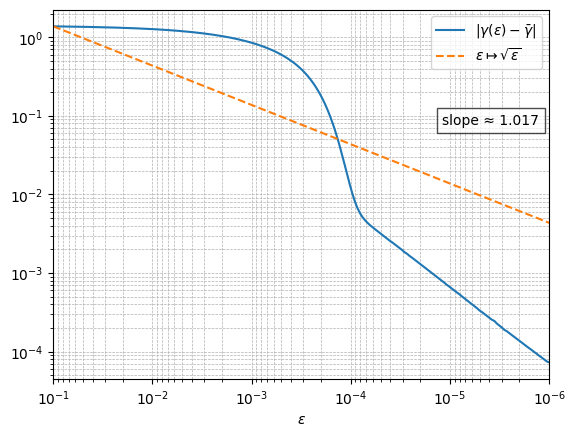}
    \end{subfigure}
    \hfill
    \begin{subfigure}{0.4\textwidth}
        \centering
        \includegraphics[width=\linewidth]{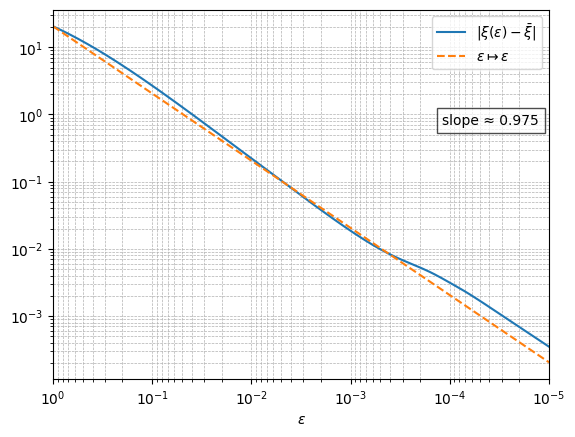}
    \end{subfigure}
    \caption{ \(\F(\cdot) = \mathrm{D}_\phi(\cdot\vert q)\) with \(\phi(t)=t(\log t - 1)\), \(\mu\) and \(\nu\) are discretized Gaussian measures. The left plot (log scale) represents the behavior of \(\varepsilon \mapsto \lVert \gamma(\varepsilon) - \bar\gamma \rVert\) (blue) compared to the expected rate (orange). The right plot (log scale) represents \(\varepsilon \mapsto \lVert \xi(\varepsilon) - \bar\xi \rVert\) (blue) compared to the expected rate (orange).}
\end{figure}

\begin{figure}[H]
    \centering
    \begin{subfigure}{0.4\textwidth}
        \centering
        \includegraphics[width=\linewidth]{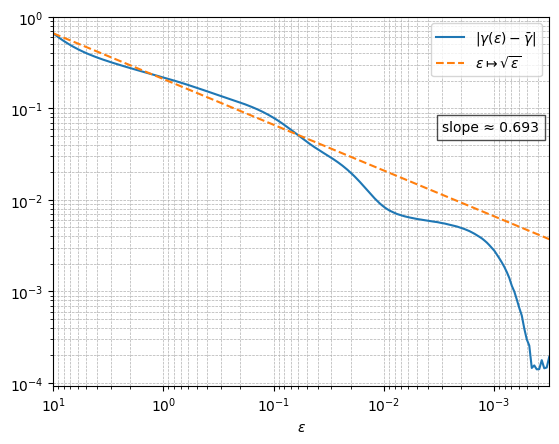}
    \end{subfigure}
    \hfill
    \begin{subfigure}{0.4\textwidth}
        \centering
        \includegraphics[width=\linewidth]{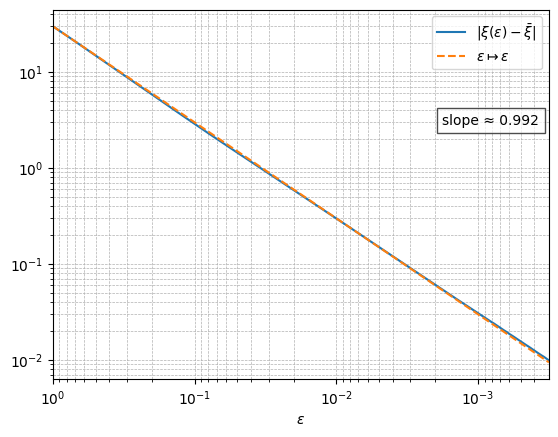}
    \end{subfigure}
    \caption{ \(\F(\cdot) = \tau \mathrm{D}_\phi(\cdot\vert q)\) with \(\phi(t)=t(\log t - 1)\) and \(\tau=10\), \(\mu\) and \(\nu\) are measures supported on \(2\)D point clouds. The left plot (log scale) represents the behavior of \(\varepsilon \mapsto \lVert \gamma(\varepsilon) - \bar\gamma \rVert\) (blue) compared to the expected rate (orange). The right plot (log scale) represents \(\varepsilon \mapsto \lVert \xi(\varepsilon) - \bar\xi \rVert\) (blue) compared to the expected rate (orange).}
\end{figure}

\begin{figure}[H]
    \centering
    \begin{subfigure}{0.4\textwidth}
        \centering
        \includegraphics[width=\linewidth]{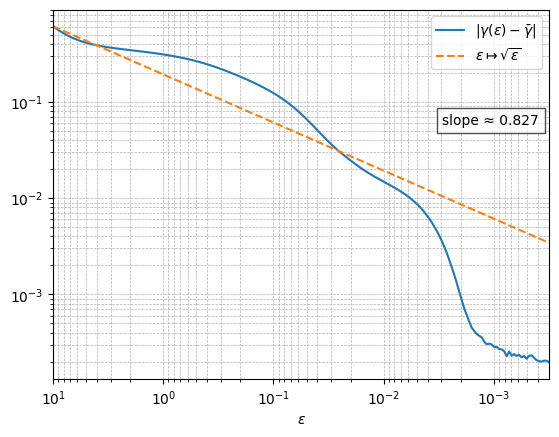}
    \end{subfigure}
    \hfill
    \begin{subfigure}{0.4\textwidth}
        \centering
        \includegraphics[width=\linewidth]{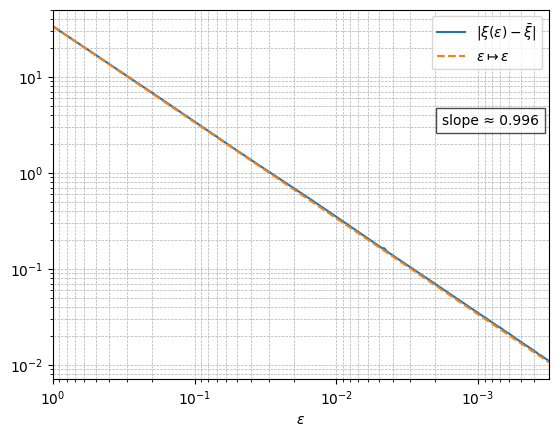}
    \end{subfigure}
    \caption{ \(\F(\cdot) = \tau\mathrm{D}_\phi(\cdot\vert q)\) with \(\phi(t)=t(\log t - 1)\) and and \(\tau=10\), \(\mu\) and \(\nu\) are measures supported on \(3\)D point clouds. The left plot (log scale) represents the behavior of \(\varepsilon \mapsto \lVert \gamma(\varepsilon) - \bar\gamma \rVert\) (blue) compared to the expected rate (orange). The right plot (log scale) represents \(\varepsilon \mapsto \lVert \xi(\varepsilon) - \bar\xi \rVert\) (blue) compared to the expected rate (orange).}
\end{figure}

We observe that the theoretical rate appears to be sharp for the dual problem, as indicated by the corresponding slopes when $\varepsilon$ is sufficiently small. In contrast, for the primal problem, the observed convergence seems to exceed the theoretical prediction, suggesting a faster rate.

\subsection{Distance to the marginals}

In the same framework as previously defined, another divergence function found in the literature is based on the Euclidean norm distance to \(q\). Taking the entropy function defined by \(\phi(x)\defeq \frac{1}{2}\vert x- 1\vert ^2 \) on \(\R^{N\!+\!M}\) leads for any \(\psi \in \R^{N\!+\!M}\) to the following divergence: 
\[
\mathrm{F}(\psi) = \frac{1}{2}\| \psi - q \|_2^2.
\]
The Legendre transform of \(\phi\) is given for any \(y \in \R\) as:
\[
\phi^*\!(y) = \frac{1}{2} y^2 + y,
\]
which is \(\mathcal{C}^2\) and strongly-convex on \(\R\).

\begin{figure}[H]
    \centering
    \begin{subfigure}{0.4\textwidth}
        \centering
        \includegraphics[width=\linewidth]{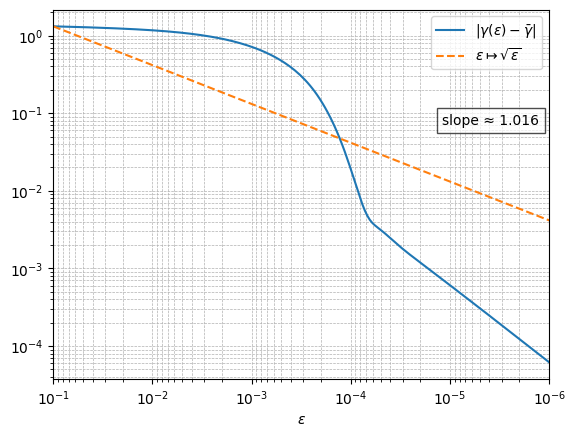}
    \end{subfigure}
    \hfill
    \begin{subfigure}{0.4\textwidth}
        \centering
        \includegraphics[width=\linewidth]{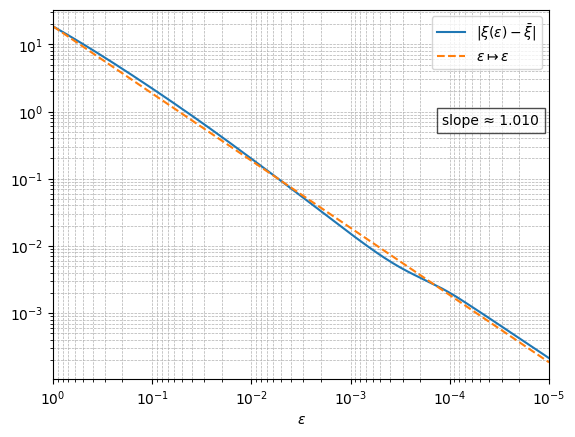}
    \end{subfigure}
   \caption{ \(\F(\cdot) = \mathrm{D}_\phi(\cdot\vert q)\) with \(\phi(t)=\frac{1}{2}\lvert t-1\rvert^2\), \(\mu\) and \(\nu\) are discretized Gaussian measures. The left plot (log scale) represents the behavior of \(\varepsilon \mapsto \lVert \gamma(\varepsilon) - \bar\gamma \rVert\) (blue) compared to the expected rate (orange). The right plot (log scale) represents \(\varepsilon \mapsto \lVert \xi(\varepsilon) - \bar\xi \rVert\) (blue) compared to the expected rate (orange).}
\end{figure}

\begin{figure}[H]
    \centering
    \begin{subfigure}{0.4\textwidth}
        \centering
        \includegraphics[width=\linewidth]{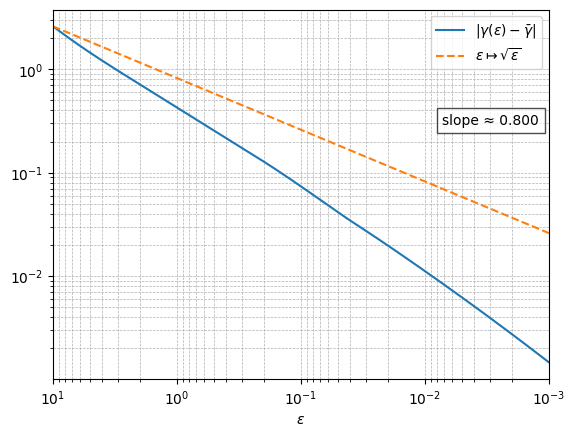}
    \end{subfigure}
    \hfill
    \begin{subfigure}{0.4\textwidth}
        \centering
        \includegraphics[width=\linewidth]{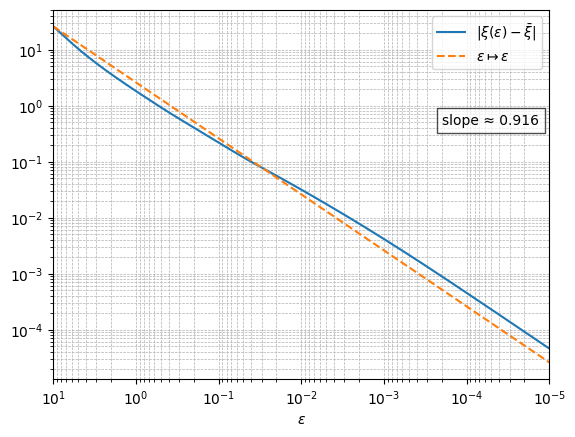}
    \end{subfigure}
    \caption{ \(\F(\cdot) = \mathrm{D}_\phi(\cdot\vert q)\) with \( \phi(t)=\frac{1}{2}\lvert t-1\rvert^2\), \(\mu\) and \(\nu\) are measures supported on \(2\)D point clouds. The left plot (log scale) represents the behavior of \(\varepsilon \mapsto \lVert \gamma(\varepsilon) - \bar\gamma \rVert\) (blue) compared to the expected rate (orange). The right plot (log scale) represents \(\varepsilon \mapsto \lVert \xi(\varepsilon) - \bar\xi \rVert\) (blue) compared to the expected rate (orange).}
\end{figure}

\begin{figure}[H]
    \centering
    \begin{subfigure}{0.4\textwidth}
        \centering
        \includegraphics[width=\linewidth]{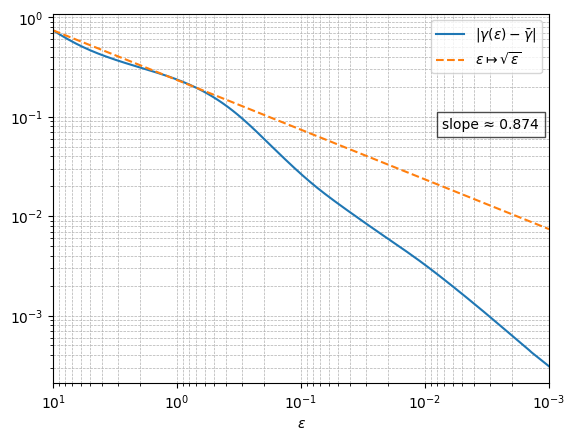}
    \end{subfigure}
    \hfill
    \begin{subfigure}{0.4\textwidth}
        \centering
        \includegraphics[width=\linewidth]{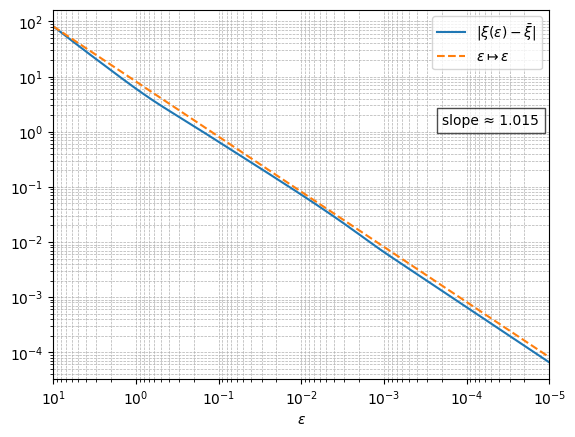}
    \end{subfigure}
    \caption{ \(\F(\cdot) = \mathrm{D}_\phi(\cdot\vert q)\) with \( \phi(t)=\frac{1}{2}\lvert t-1\rvert^2\), \(\mu\) and \(\nu\) are measures supported on \(2\)D point clouds. The left plot (log scale) represents the behavior of \(\varepsilon \mapsto \lVert \gamma(\varepsilon) - \bar\gamma \rVert\) (blue) compared to the expected rate (orange). The right plot (log scale) represents \(\varepsilon \mapsto \lVert \xi(\varepsilon) - \bar\xi \rVert\) (blue) compared to the expected rate (orange).}
\end{figure}

In this example, we see that the theoretical rate remains sharp for the dual problem, but once again does not accurately describe the behavior of the primal problem.

\smallskip

\noindent{\textbf{Acknowledgments.}} The authors acknowledge  the support of the FMJH Program PGMO and the ANR project GOTA (ANR-23-CE46-0001).

\printbibliography

\end{document}